\documentclass[a4paper,11pt]{article}
\usepackage{indentfirst, latexsym, bm}
\usepackage[numbers,sort&compress]{natbib}
\usepackage{epsfig}
\usepackage{amsthm}
\usepackage{amsmath}
\usepackage{amsbsy}
\usepackage{amsfonts}
\usepackage{graphicx}
\usepackage{enumerate}

\newtheorem{theorem}{Theorem}[section]
\newtheorem{cor}{Corollary}[section]
\newtheorem{lem}{Lemma}[section]
\newtheorem{prop}{Proposition}[section]
\theoremstyle{definition}
\newtheorem{defn}{Definition}[section]
\theoremstyle{remark}
\newtheorem{rem}{Remark}[section]
\numberwithin{equation}{section}

\newcommand{\abs}[1]{\left\vert#1\right\vert}

\newcommand{\norm}[1]{\left\Vert#1\right\Vert}

\begin{document}

\title{A Posteriori Error Estimates of Krylov Subspace Approximations
to Matrix Functions\thanks{Supported in part by
National Basic Research Program of China  2011CB302400 and National
Science Foundation of China (Nos. 11071140, 11371219).}}
\author{Zhongxiao Jia\thanks{Department of Mathematical
Sciences, Tsinghua University, Beijing, 100084,
People's Republic of China, jiazx@tsinghua.edu.cn}
\and Hui Lv\thanks{Department of Mathematical Sciences,
Tsinghua University, Beijing, 100084,
People's Republic of China, lvh09@mails.tsinghua.edu.cn}}

\date{}

\maketitle
\begin{abstract}
Krylov subspace methods for approximating a matrix
function $f(A)$ times a vector $v$ are analyzed in this paper.
For the Arnoldi approximation to $e^{-\tau A}v$, two reliable
a posteriori error estimates are derived from the new bounds and
generalized error expansion we establish.
One of them is similar to the residual norm of an approximate solution of the
linear system, and the other one is determined critically by the
first term of the error expansion of the Arnoldi approximation to
$e^{-\tau A}v$ due to Saad. We prove that each of the two estimates is reliable
to measure the true error norm, and the second one theoretically
justifies an empirical claim by Saad. In the paper, by introducing certain
functions $\phi_k(z)$ defined recursively by the given function $f(z)$ for
certain nodes, we obtain the error
expansion of the Krylov-like approximation
for $f(z)$ sufficiently smooth, which generalizes Saad's result
on the Arnoldi approximation to $e^{-\tau A}v$. Similarly, it is shown
that the first term of the generalized error expansion can be used as
a reliable a posteriori estimate for the Krylov-like approximation to
some other matrix functions times $v$.
Numerical examples are reported to demonstrate the effectiveness of
the a posteriori error estimates for the Krylov-like approximations to
$e^{-\tau A}v$, $\cos(A)v$ and $\sin(A)v$.
\smallskip


\smallskip


\end{abstract}

\section{Introduction}

For a given matrix $A\in \mathbb{C}^{N\times N}$ and a
complex-valued function $f(z)$, the problem of numerically approximating
$f(A)v$ for a given vector $v\in\mathbb{C}^N$
arises in many applications; see, e.g., \cite{rd1,arnoldi1,book3,Philippe,other2}.
Solving the linear system $Ax=b$, which involves the reciprocal function $f(z)=z^{-1}$,
is a special instance and of great importance. Approximating the
matrix exponential $e^{A}v$ is the core of many exponential integrators
for solving systems of linear ordinary  differential equations or
time-dependent partial differential equations.
The trigonometric matrix functions acting on a vector often arise in
the solution of the second-order differential problems.
Other applications include the evaluation of $f(A)v$ for the square root function,
the sign function, the logarithm function, etc.

The matrix $A$ is generally large and sparse in many applications,
for which computing $f(A)$ by conventional algorithms for small or
medium sized $A$ is unfeasible. In this case, Krylov subspace
methods, in which the only action of $A$ is to form matrix-vector products,
have been effective tools for computing $f(A)v$. The Krylov
subspace methods for approximating $f(A)v$ first project $A$ onto a
low dimensional Krylov subspace $\mathcal{K}_m(A,v)=
span\{v,\,Av,\,\ldots,\,A^{m-1}v\}$, and then compute an approximation
to $f(A)v$ by calculating a small sized matrix function times a specific
vector \cite{arnoldi1,book3,arnoldi2}. Since computing an approximation
to $f(A)v$ needs all the basis vectors of $\mathcal{K}_m(A,v)$,
this generally limits the dimension of $\mathcal{K}_m(A,v)$ and makes
restarting necessary. The restarted Krylov subspace methods presented
in \cite{restarted1,restarted8,restarted10,restarted4,restarted5,error4}
aim at overcoming this disadvantage. However, restarting may slow
down the convergence substantially, and may even fail to converge
\cite{restarted2,restarted3,restarted8,error4}. A deflated
restarting technique is proposed in \cite{deflaterestarted1} that is
adapted from \cite{deflaterestarted2,deflaterestarted3} for eigenvalue
problems and linear systems so as to accelerate the convergence.

A priori error estimates for Krylov subspace approximations
to $f(A)v$ can be found in a number of papers, e.g.,
\cite{error5,polykry3,polykry4,polykry5}.
As it is known, the convergence of the approximations depends on the
regularity of the function $f(z)$ over the given
domain, the spectral properties of $A$ and
the choice of the interpolation nodes \cite{polykry4}.
In particular, there have been several a priori error estimates available
for Krylov subspace methods to approximate $e^{-\tau A}v$;
see, e.g., \cite{rd3,arnoldi1,arnoldi2}. As it has turned out,
the methods may exhibit the superlinear convergence \cite{error2}. A few
quantitatively similar bounds have been derived in \cite{error1,error7}.
The results in \cite{error3} also reveal the superlinear convergence
of the approximations, but they are less sharp than those in
\cite{error1,error2,error7}, at least experimentally \cite{rd4}.
The  bounds in \cite{error8,doc1} relate the convergence
to the condition number of $A$ when it is symmetric
positive definite, and can be quite accurate. Obviously,
these results help to understand the methods, but are not applicable in practical
computations as they involve the spectrum or the field of
values of $A$.

For many familiar linear algebra problems, such as the linear system,
the eigenvalue problem, the least squares problem, and the singular value
decomposition, there is a clear residual notion that can be
used for determining the convergence and designing stopping criteria
in iterative methods. However,
the Krylov subspace approximation to $f(A)v$ is not naturally equipped
with a standard residual notion. In fact, the lack of a residual notion is
a central problem in computing matrix functions acting on a vector.
Effective a posteriori error estimates are,
therefore, very appealing and crucial to design practical algorithms.
Effective terminations of iterative algorithms have attracted much attention
over the last two decades. In \cite{other5,polykry4}, the authors have
introduced a generalized residual of the Arnoldi approximation to $f(A)v$.
In \cite{botchev}, the authors have provided more details when $f(z)=e^{z}$.
For $A$ symmetric, a Krylov subspace method based on the Chebyshev series expansion
for the exponential function in the spectral interval of $A$ has been proposed
in \cite{polykry1}.
The method uses very low storage, but requires accurate approximation of
the spectral interval of $A$. An a posteriori error estimate for this method
has been derived in \cite{polykry1}. In \cite{polykry2}, a posteriori error
estimates have been established for several polynomial Krylov approximations to a
class of matrix functions, which are formulated in the general integral form,
acting on a vector. These functions include
exponential-like functions, cosine, and the sinc function
arising in the solution of second-order differential problems. These
estimates cannot be unified, and each of them is function dependent and
quite complicated. In \cite{arnoldi2}, Saad established an error expansion
for the Arnoldi approximation to $e^{A}v$, which is exactly our later expansion
(\ref{saadbound}) for $\tau=-1$ in Corollary \ref{error expansion2}.
Empirically, he claimed that
the first term of the error expansion can be a reliable a posteriori
error estimate for the Arnoldi approximation to $e^{A}v$; see later
(\ref{criter}).
However, its theoretical justification has not yet been given hitherto.

In this paper, by introducing certain functions $\phi_k(z)$ defined
recursively by the given function $f(z)$ for certain nodes,
we obtain an error expansion of the Krylov-like approximation
to $f(A)v$ for $f(z)$ sufficiently smooth. The error expansion is an infinite
series, which generalizes Saad's result on the Arnoldi approximation to
$e^{A}v$. With a specific choice of the nodes, the infinite
series reduces to a finite term series.  We establish two new
upper bounds for the Arnoldi approximation to $e^{-\tau A}v$ where $\tau$ is
often the time step parameter in a finite difference time-stepping method.
In the case that $A$ is Hermitian, we derive more compact results that are
more convenient to use. For $e^{-\tau A}v$, by a rigorous
analysis of the infinite series expansion of the error, we derive a reliable a
posteriori error estimate, theoretically proving
that the first term in the infinite series expansion generally suffices to provide a
good estimate for the true error and confirming the empirical claim due to
Saad \cite{arnoldi2}. We also show why the first term of the error expansion
provides a new reliable a posteriori estimate for the Krylov-like approximation
when $f(z)$ is sufficiently smooth.
As a consequence, the Krylov-like approximations to $e^{-\tau A}v$ and more
functions acting on a vector are equipped with theoretically effective stopping
criteria. Some typical numerical
examples are reported to illustrate the effectiveness of our a posteriori
error estimates for the Krylov-like approximations to
not only $e^{-\tau A}v$ but also $\sin(A)v$ and $\cos(A)v$.

The paper is organized as follows. In Section~\ref{sec:2}, we review some
definitions and properties of matrix functions. We also describe
the framework and some basic properties of Krylov subspace
approximations to $f(A)v$ based on the Krylov-like decomposition. In
Section~\ref{sec:3}, we establish the error expansion of the Krylov-like
approximation to $f(A)v$ for $f(z)$ sufficiently smooth, which includes
$f(z)=e^z$ as a special case. In Section~\ref{sec:4},
we present some upper bounds for the error of the Arnoldi approximation to
$e^{-\tau A}v$ and derive two more compact bounds in the case that
$A$ is Hermitian. In Section~\ref{sec:5}, we derive some a posteriori error
estimates from the bounds and expansions we establish,
and justify the rationale of the first term of each error
expansion as an error estimate. The numerical results are then reported to
illustrate the sharpness of our a posteriori error estimates for
the matrix exponential, the sine and cosine functions. Finally,
in Section~\ref{sec:6}  we conclude the paper with some remarks.

Throughout the paper let $A$ be a given matrix of order $N$, denote by
$\norm{\cdot}$ the Euclidean vector norm and the
matrix 2-norm, by the asterisk $*$ the conjugate transpose of a
matrix or vector, and by the superscript $T$ the transpose of a matrix or vector.
The set $F(A)\equiv \{x^*Ax:x\in\mathbb{C}^N,\ x^*x=1\}$
denotes the field of values of $A$, $spec(A)$ is the spectrum of $A$,
and $\mathbf{0}$ denotes a zero matrix with appropriate size.

\section{Krylov subspace approximations to $f(A)v$} \label{sec:2}
In this section, we review some definitions and properties of matrix
functions to be used and present a class of popular Krylov subspace
approximations to $f(A)v$.

\subsection{Polynomial interpolatory properties of $f(A)$} \label{subsec:2.1}

The matrix function $f(A)$ can be equivalently defined via
the Jordan canonical form, the Hermite interpolation or the
Cauchy integral theorem; see \cite[p. 383-436]{book2} for
details. We review the definition via the Hermite interpolation and some
properties associated with it.

\begin{defn}\label{def}
Let $A$ have the minimal polynomial
$q_A(z)=(z-\lambda_1)^{r_1} \cdots (z-\lambda_{\mu})^{r_{\mu}},$
where $\lambda_1,\ldots,\lambda_{\mu}$ are distinct and all $r_i\geq
1$, let $f(z)$ be a given scalar-valued function whose domain includes
the points $\lambda_1,\ldots,\lambda_{\mu}$, and assume that each
$\lambda_i$ is in the interior of the
domain and $f(z)$ is $(r_i-1)$ times differentiable at
$\lambda_i$. Then $f(A)\equiv p(A),$ where $p(z)$ is the unique polynomial
of degree $\sum_{i=1}^{\mu}r_i-1$ that satisfies the interpolation conditions
\begin{equation*}
  p^{(j)}(\lambda_i)=f^{(j)}(\lambda_i),\ j=0,1,\ldots,r_i-1,\ i=1,\ldots,\mu.
\end{equation*}
\end{defn}

\begin{prop}\label{prop1}
The polynomial $p(z)$ \textit{interpolating $f(z)$ and its
derivatives at the roots of $q_A(z)=0$} can be given explicitly by
the Hermite interpolating polynomial. Its Newtonian divided difference
form is
\begin{eqnarray*}
p(z)& = &f[x_1] + f[x_1,x_2](z-x_1) + f[x_1,x_2,x_3](z-x_1)(z-x_2) + \cdots \nonumber\\
 & & + f[x_1,x_2,\ldots,x_m](z-x_1)(z-x_2)\cdots(z-x_{m-1}),
\end{eqnarray*}
where $m=deg\ q_A(z)$, the set $\{x_i\}_{i=1}^m$ comprises the
distinct eigenvalues $\lambda_1,\ldots,\lambda_\mu$ with $\lambda_i$
having multiplicity $r_i$, and $f[x_1,x_2,\ldots,x_k]$ is the divided difference of
order $k-1$ at $x_1,x_2,\ldots,x_k$ with $f[x_1]=f(x_1)$.
\end{prop}

Bounds for $\norm{f(A)}$ are useful for both theoretical and practical purposes.
The following bound, a variant of Theorem 4.28 in \cite[p. 103]{book3}, is
derived in terms of the Schur decomposition .

\begin{theorem}\label{boundf}
Let $Q^*AQ=D+U$ be a Schur
decomposition of $A$, where $D$ is diagonal and $U$ is strictly upper triangular.
If $f(z)$ is analytic
on a closed convex set $\Omega$ containing $spec(A)$. Then
\begin{equation}\label{boundf1}
    \norm{f(A)}\leq\sum_{i=0}^{n-1} \sup_{z\in\Omega}
    \abs{f^{(i)}(z)}\frac{\norm{U}_F^i}{i!},
\end{equation}
where $\norm{U}_F$ denotes the Frobenius norm of $U$.
\end{theorem}

\subsection{The Krylov-like decomposition and computation of $f(A)v$}
\label{subsec:2.2}

The Arnoldi approximation to $f(A)v$ is based on the Arnoldi decomposition
\begin{equation}\label{Arnoldi decomposition}
AV_m=V_mH_m+h_{m+1,m}v_{m+1}e_m^{T},
\end{equation}
where the columns of $V_m=[v_1,\,v_2,\,\ldots,\,v_m]$ form an
orthonormal basis of the Krylov subspace $\mathcal{K}_m(A,v)=
span\{v,\,Av,\,\ldots,\,A^{m-1}v\}$ with $v_1=v/\norm{v}$, $H_m=[h_{i,j}]$
is an unreduced upper Hessenberg matrix and $e_m\in \mathbb{R}^m$
denotes the $m$th unit coordinate vector. The Arnoldi approximation
to $f(A)v$ is given by
\begin{equation}\label{Arnoldi approximation}
f_m=\norm{v}V_mf(H_m)e_1.
\end{equation}

For $H_m$, there is the following well-known property \cite{other4}.

\begin{prop}\label{prop2}
Each eigenvalue of $H_m$ has geometric multiplicity equal to one, and
the minimal polynomial of $H_m$ is its characteristic polynomial.
\end{prop}

In \cite{deflaterestarted1}, a more general decomposition than
(\ref{Arnoldi decomposition}), called
the Krylov-like decomposition, is introduced, and the
the associated Krylov-like approximation to $f(A)v$ is given.
The Krylov-like decomposition of $A$ with respect to $\mathcal{K}_m(A,v)$
is of the form

\begin{equation}\label{Krylov decomposition}
    AW_{m+l}=W_{m+l}K_{m+l}+wk_{m+l}^{T},
\end{equation}
where $K_{m+l}\in \mathbb{C}^{(m+l)\times (m+l)},$
$W_{m+l}\in \mathbb{C}^{N\times (m+l)}$ with $range(W_{m+l})=\mathcal{K}_m(A,v)$,
$w\in \mathcal{K}_{m+1}(A,v)\setminus \mathcal{K}_m(A,v)$ and
$k_{m+l}\in \mathbb{C}^{m+l}$.

Let $f(z)$ be a function such that $f(K_{m+l})$ is defined. Then the Krylov-like
approximation to $f(A)v$ associated with (\ref{Krylov decomposition}) is given by
\begin{equation}\label{Krylov approximation}
    \hat{f}_m=W_{m+l}f(K_{m+l})\hat{b},
\end{equation}
where $\hat{b}\in \mathbb{C}^{m+l}$ is any vector such that $W_{m+l}\hat{b}=v.$

Since the vector $w$ lies in $\mathcal{K}_{m+1}(A,v)\setminus \mathcal{K}_m(A,v)$,
it can be expressed as $w=p_m(A)v$ with a unique polynomial $p_m(z)$ of exact degree
$m$. The following result is proved in \cite{deflaterestarted1} for $\hat {f}_m$.

\begin{theorem}\label{property1}
For the polynomial $p_m(z)$ defined by $w=p_m(A)v$, the Krylov-like approximation
(\ref{Krylov approximation}) to $f(A)v$ can be characterized as
$\hat{f}_m=q_{m-1}(A)v$, where $q_{m-1}(z)$ interpolates the function $f(z)$ in the
Hermitian sense at zeros of $p_m(z)$, i.e., at some but, in general, not all
eigenvalues of $K_{m+l}$.
\end{theorem}
For more properties of the Krylov-like approximation to $f(A)v$, one can
refer to \cite{deflaterestarted1}. Specifically,
the Krylov-like decomposition (\ref{Krylov decomposition}) includes
the following three important and commonly used decompositions:
\begin{itemize}
  \item The Arnoldi decomposition if $l=0$ and the columns of $W_m$ are
  orthonormal and form an ascending basis of $\mathcal{K}_m(A,v)$, that is, the
  first $j$ columns of $W_m$ generate $\mathcal{K}_j(A,v)$,  and $K_m$ is
  upper Hessenberg.

  \item The Arnoldi-like decomposition that corresponds to
  the restarted Arnoldi approximation to $f(A)v$
  if $l=0$ and the columns of $W_m$ form an ascending basis of $\mathcal{K}_m(A,v)$,
  and $K_m$ is upper Hessenberg; see \cite{restarted1,restarted8} for details.

  \item The Krylov-like decomposition that corresponds to a restarted Arnoldi
  method with deflation, where an $l$ dimensional approximate invariant
  subspace of $A$ is augmented to an $m$ dimensional
  Krylov subspace at each restart \cite{deflaterestarted1}.
\end{itemize}

Define the error
\begin{equation}\label{error}
E_m(f)\equiv f(A)v-W_{m+l}f(K_{m+l})\hat{b}.
\end{equation}
Unlike those basic linear algebra problems, such as the linear system,
the eigenvalue problem, the least squares problem and the singular value
decomposition, which have exact a posteriori residual norms that
are used for stopping criteria in iterative methods,
the Krylov subspace approximation to $f(A)v$ is not naturally equipped
with a stopping criterion since when approximating $f(A)v$
there is no immediate quantity analogous to
the residual in the mentioned basic linear algebra problems and
the error norm $\norm{E_m(f)}$ cannot be computed explicitly. Therefore, it is
crucial to establish reliable and accurate a posteriori error
estimates in practical computations. As it has turned out in the literature,
this task is nontrivial. We will devote ourselves to this task in the sequel.

\section{The expansion of the error $E_m(f)$} \label{sec:3}

In this section, we analyze the error $E_m(f)$ produced by the Krylov-like
approximation (\ref{Krylov approximation}). Inspired by the error expansion
derived by Saad \cite{arnoldi2} for the Arnoldi approximation to $e^Av$, which
is our later (\ref{saadbound}) when $\tau=-1$ , we establish a more general
form of the error expansion for all sufficiently smooth functions $f(z)$.
Our result differs from Saad's in that
his expansion formula is only for the Arnoldi approximation for $f(z)=e^z$
and is expressed in terms of $A^kv_{m+1},\
k=0,1,\ldots, \infty$, ours is more general, insightful and informative and
applies to the Krylov-like approximations for all sufficiently smooth functions
$f(z)$.

Throughout the paper, assume that $f(z)$ is analytic in a closed
convex set $\Omega$ and on its boundary, which contains
the field of values $F(A)$ and the field of
values $F(K_{m+l})$. Let the sequence
$\{z_i\}_{i=0}^{\infty}$ belong to the set $\Omega$ and
be ordered so that equal points are contiguous, i.e.,
\begin{equation} \label{order}
z_i=z_j\ (i<j)\ \rightarrow z_i=z_{i+1}=\cdots=z_j,
\end{equation}
and define the function sequence $\{\phi_{k}\}_{k=0}^{\infty}$ by the recurrence
\begin{equation}\label{phi}
\left\{\begin{array}{rcl}
\phi_0(z) & = & f(z),\\
\phi_{k+1}(z) & = & \dfrac{\phi_k(z)-\phi_k(z_k)}{z-z_k},\ k \geq 0.
\end{array}
\right.
\end{equation}
Noting that $\phi_k(z_k)$ is well defined by continuity,
it is clear that these functions are analytic for all $k$.

Let $f[z_0,\,z_1,\ldots,z_{k}]$ denote the $k$th divided differences of $f(z)$.
If $f(z)$ is $k$-times continuously differentiable, $f[z_0,\,z_1,\ldots,z_{k}]$
is a continuous function of its arguments. Moreover, if
$\Omega$ is a closed interval in the real axis, we have
\begin{equation}\label{diff1}
 f[z_0,\,z_1,\ldots,z_{k}] = \dfrac {f^{(k)}(\zeta)}{k!}\ \text{for some}\
 \zeta\in \Omega.
\end{equation}
However, no result of form (\ref{diff1}) holds for complex $z_i$. Nevertheless,
if $z_0,\,z_1,\ldots,z_{k}\in \Omega$, then
it holds \cite[p. 333]{book3} that
\begin{equation}\label{Cdiff}
 \abs{f[z_0,\,z_1,\ldots,z_{k}]} \leq \frac{\max_{z\in \Omega}\abs{f^{(k)}(z)}}{k!}.
\end{equation}
From the above it is direct to get
\begin{equation}
\phi_{k+1}(z)=f[z,\,z_0,\ldots,z_k]. \label{phifun}
\end{equation}
Next we establish a result on the expansion of the error $E_m(f)$.

\begin{theorem}\label{error expansion1}
Assume that $f(z)$ is analytic in the closed convex set $\Omega$ and on its
boundary, which contains the field of values $F(A)$ and the field of
values $F(K_{m+l})$,
and there exists a positive constant $C$ such that $\max_{z\in \Omega}
\abs{f^{(k)}(z)}\leq C$ for
all $k\geq 0$. Then the error $E_m(f)$ produced by the Krylov-like approximation
satisfies the expansion
\begin{equation}\label{expansion1}
    E_m(f)=f(A)v-W_{m+l}f(K_{m+l})\hat{b}=\sum_{k=1}^{\infty}k_{m+l}^T
    \phi_k(K_{m+l})\hat{b}\ q_{k-1}(A)w,
\end{equation}
where $q_0(z)=1,\ q_k(z)=(z-z_0)\cdots(z-z_{k-1}),\ k\geq 1$, and
$z_i\in \Omega$ for all $i\geq0,$
and $\hat{b}\in \mathbb{C}^{m+l}$ is any vector satisfying $W_{m+l}\hat{b}=v.$

In particular, if $z_0,z_1,\ldots,z_{N-1}$ are the $N$ exact eigenvalues of $A$
counting multiplicities, then the infinite series
{\rm (\ref{expansion1})} simplifies to a finite one
\begin{equation}
 E_m(f)=f(A)v-W_{m+l}f(K_{m+l})\hat{b}=\sum_{k=1}^{N}k_{m+l}^T
    \phi_k(K_{m+l})\hat{b}\ q_{k-1}(A)w. \label{finite}
\end{equation}
\end{theorem}

\begin{proof}
Defining
\begin{equation}\label{defsm}
s_m^{(j)}=\phi_j(A)v-W_{m+l}\phi_j(K_{m+l})\hat{b},
\end{equation}
which is the error of the Krylov-like approximation to $\phi_j(A)v$,
and making use of
the relation $f(z)=(z-z_0)\phi_1(z)+f(z_0)$, we have
\begin{eqnarray}
  f(A)v &=& f(z_0)v+(A-z_0I)\phi_1(A)v \nonumber\\
        &=& f(z_0)v+(A-z_0I)(W_{m+l}\phi_1(K_{m+l})\hat{b}+s_m^{(1)}) \nonumber \\
        &=& f(z_0)v+\left(W_{m+l}(K_{m+l}-z_0I)+wk_{m+l}^T\right)\phi_1(K_{m+l})
        \hat{b}+(A-z_0I)s_m^{(1)} \nonumber \\
        &=& W_{m+l}\left(f(z_0)\hat{b}+(K_{m+l}-z_0I)\phi_1(K_{m+l})\hat{b}\right)
        \nonumber\\
        & &+k_{m+l}^T\phi_1(K_{m+l})\hat{b}\ w+(A-z_0I)s_m^{(1)} \nonumber\\
        &=& W_{m+l}f(K_{m+l})\hat{b}+k_{m+l}^T\phi_1(K_{m+l})\hat{b}\ w
        +(A-z_0I)s_m^{(1)}.
        \label{s1}
\end{eqnarray}
Proceeding in the same way, we obtain also
\begin{eqnarray*}
  \phi_1(A)v &=& \phi_1(z_1)v+(A-z_1I)\phi_2(A)v \nonumber\\
        &=& \phi_1(z_1)v+(A-z_1I)(W_{m+l}\phi_2(K_{m+l})\hat{b}+s_m^{(2)})
        \nonumber \\
        &=& \phi_1(z_1)v+\left(W_{m+l}(K_{m+l}-z_1I)+wk_{m+l}^T\right)
        \phi_2(K_{m+l})\hat{b}+(A-z_1I)s_m^{(2)} \nonumber \\
        &=& W_{m+l}\left(\phi_1(z_1)\hat{b}+(K_{m+l}-z_1I)\phi_2(K_{m+l})
        \hat{b}\right) \nonumber\\
        & &+k_{m+l}^T\phi_2(K_{m+l})\hat{b}\ w+(A-z_1I)s_m^{(2)} \nonumber\\
        &=& W_{m+l}\phi_1(K_{m+l})\hat{b}+k_{m+l}^T\phi_2(K_{m+l})\hat{b}\
        w+(A-z_1I)s_m^{(2)}.
\end{eqnarray*}
By definition (\ref{defsm}), this gives $s_m^{(1)}=k_{m+l}^T\phi_2(K_{m+l})\hat{b}\ w
+(A-z_1I)s_m^{(2)}$. Continue expanding $s_m^{(2)},\ s_m^{(3)},\ldots,s_m^{(j-1)}$
in the same manner. Then we get the following key recurrence formula
$$
s_m^{(j-1)}=k_{m+l}^T\phi_j(K_{m+l})\hat{b}\ w+(A-z_{j-1} I)s_m^{(j)},\
j=2,\ldots,\infty.
$$
Substituting $s_m^{(1)},\ s_m^{(2)},\ldots,s_m^{(j-1)}$ successively into
(\ref{s1}), we obtain
\begin{eqnarray*}
  E_m(f) &=& f(A)v-W_{m+l}f(K_{m+l})\hat{b}\nonumber \\
      &=& \sum_{k=1}^{j}k_{m+l}^T\phi_k(K_{m+l})\hat{b}\ q_{k-1}(A)w + q_j(A)s_m^{(j)},
\end{eqnarray*}
where $q_0(z)=1,\ q_k(z)=(z-z_0)\cdots(z-z_{k-1}),\ k\geq 1$.

Next we prove that $\|q_j(A)s_m^{(j)}\|$ converges to zero faster than $O(1/j)$ for
$j\ge M$ with $M$ a sufficiently large positive integer,
i.e., $\|q_j(A)s_m^{(j)}\|\leq o(1/j)$,
when $j$ is large enough. According to (\ref{Cdiff}), we get
\begin{equation}\label{phij}
   \mid \phi_j(z)\mid=\mid f[z,\,z_0,\,\ldots,\,z_{j-1}]\mid\leq
   \frac{\max_{\zeta\in \Omega} \abs{f^{(j)}(\zeta)}}{j!}
    \leq \frac{C}{j!}.
\end{equation}
Let $Q_1^*AQ_1=D_1+U_1$ and $Q_2^*K_{m+l}Q_2=D_2+U_2$ be the Schur decompositions
of $A$ and $K_{m+l}$, where $U_1$ and $U_2$ are strictly upper triangular.
By the assumptions on $\Omega$,
$F(A)$ and $F(K_{m+l})$,
it follows from (\ref{boundf1}) and (\ref{phij}) that
$$\norm{\phi_j(A)}\leq \frac{C}{j!}\sum_{i=0}^{N-1}\frac{\norm{U_1}_F^i}{i!},\
\norm{\phi_j(K_{m+l})}\leq \frac{C}{j!}\sum_{i=0}^{m+l-1}\frac{\norm{U_2}_F^i}{i!}.$$
Therefore, by the above and (\ref{defsm}), it holds that
\begin{equation}\label{sj}
  \norm{s_m^{(j)}}\leq \|\phi_j(A)\|\|v\|+\|W_{m+l}\|\|\phi_j(K_{m+l})\|\|\hat b\|
  \leq \frac{C_1}{j!},
\end{equation}
where
\begin{equation}\label{C1}
    C_1=C\left(\norm{v}\sum_{i=0}^{N-1}\frac{\norm{U_1}_F^i}{i!}
+\norm{W_{m+l}}\|\hat b\|\sum_{i=0}^{m+l-1}\frac{\norm{U_2}_F^i}{i!}\right).
\end{equation}

Since $z_0,z_1,\ldots,z_{j-1}\in \Omega$, we have $\|A-z_i I\|\leq \|A\|+
| z_i |\leq C_2$,
$i=0,1,\ldots,j-1$, where $C_2$ is a bounded constant.
Therefore, we have $\|q_j(A)\|\leq C_2^{j}$. So, by Stirling's inequality
(see, e.g., \cite[p. 257]{Striling})
$$
\sqrt{2\pi j}\left(\frac{j}{e}\right)^j<j!<\sqrt{2\pi j}
\left(\frac{j}{e}\right)^je^{\frac{1}{12j}},
$$
where $e$ is the base number of natural logarithm, we obtain
\begin{equation}
\|q_j(A)s_m^{(j)}\|\leq  C_1\frac{C_2^{j}}{j!}<\frac{C_1}{\sqrt{2\pi j}}
\left(\frac{C_2e}{j}\right)^{j}. \label{pjsm}
\end{equation}
So $\|q_j(A)s_m^{(j)}\|$ tends to zero swiftly as $j$ increases, faster than
$1/j^{3/2}$ when $j\ge M$, where $M$ is the positive integer making
the second factor of the right-hand side (\ref{pjsm}) smaller than $1/j$. We
remark that this factor itself converges to zero very fast as $j$ increases,
and essentially can be made smaller than $1/j^\alpha$ with an arbitrarily
given constant $\alpha \ge 1$ once $j$ is large enough. Therefore, we have
\begin{equation*}
    E_m(f)=f(A)v-W_{m+l}f(K_{m+l})\hat{b}=\sum_{k=1}^{\infty}k_{m+l}^T
    \phi_k(K_{m+l})\hat{b}\ q_{k-1}(A)w,
\end{equation*}
which is just (\ref{expansion1}).

If $z_0,z_1,\ldots,z_{N-1}$ are the $N$ exact eigenvalues of $A$ counting
multiplicities, then, by the Cayley-Hamilton theorem, $q_{N-1}(A)=0$,
from which it follows that $q_k(A)=0$ for $k\ge N$, independent of $z_N,
z_{N+1},\ldots$. Therefore, from (\ref{expansion1}) we get (\ref{finite})
without any requirement on the size of $| f^{(k)}(z) |$ over $\Omega$ for $k\geq 1$.
\end{proof}

From now on, to be specific for $f(z)=e^z$, we instead use the notation
$E_m(e^z,\tau)$ to denote
the error of the Arnoldi approximation to $e^{-\tau A}v$:
$$
E_m(e^z,\tau) = e^{-\tau A}v-\norm{v}V_me^{-\tau H_m}e_1.
$$
Suppose that the origin is in the set $\Omega$. Then by taking
all $z_k=0$ and $f(z)=e^{z}$, (\ref{expansion1}) reduces to the following
form, which simplifies to the error expansion due to
Saad \cite[Theorem 5.1]{arnoldi2} when $\tau=-1$.

\begin{cor}\label{error expansion2}
With the notation described previously, the error
produced by the Arnoldi approximation \rm(\ref{Arnoldi approximation}) to
$e^{-\tau A}v$ satisfies
\begin{equation}
  E_m(e^z,\tau) = -\tau \norm{v}h_{m+1,m}\sum_{k=1}^{\infty}e_m^T
  \phi_k(-\tau H_m)e_1(-\tau A)^{k-1}v_{m+1}. \label{saadbound}
\end{equation}
\end{cor}

\begin{rem}\label{rem3.1}
In comparison with Corollary~\ref{error expansion2}, (\ref{expansion1})
has two distinctive features. First, the theorem holds for more
general analytic functions other than only the exponential function. Second,
it holds for the Krylov-like decomposition, a generalization
of the Arnoldi decomposition. As a consequence, the Arnoldi approximation
\cite{arnoldi1,arnoldi2}, the restarted Krylov subspace approach
\cite{restarted1,restarted8} and the deflated restarting approach
\cite{deflaterestarted1} for approximating $f(A)v$ all have the error of
form (\ref{expansion1}).
\end{rem}

\begin{rem}\label{rem3.2}
Apart from $e^z$, (\ref{expansion1}) applies to the trigonometric functions
$\cos(z)$ and $\sin(z)$ as well. It is seen from the proof
that $\|q_k(A)s_m^{(k)}\|$ decays swiftly, faster than $1/k^{3/2}$ for
$k\ge M$ with $M$ a suitable positive integer. Let us
look into the size of $M$. For brevity, we suppose all $z_i\in F(A)$,
so that $| z_i |\leq \|A\|$ and we can take the constant $C_2=2\|A\|$.
Therefore, from (\ref{pjsm}), such $M$
is the minimal $k$ making
$$
\left(\frac{k}{2\|A\|e}\right)^k>k.
$$
From this, we get
$$
\log k>\log (2\|A\|e)+\frac{\log k}{k},
$$
where $\log(\cdot)$ is the logarithm of base number 10.
Note that $0\leq \frac{\log k}{k}<1$ for $k\ge 1$. Therefore, we
relax the above inequality problem to
$$
\log k\geq \log (2\|A\|e)+1=\log (20\|A\|e),
$$
whose minimal $k=M=\left\lceil 20\|A\|e \right \rceil$, where
$\lceil \cdot \rceil$ is the ceil function.
From this, the sum of norms of the terms from the $M$th
to infinity is no more than the order of $\sum_{M}^{\infty}k^{-3/2}\leq
\int_{M-1}^{\infty}x^{-3/2}dx=\frac{2}{\sqrt{M-1}}$, which is smaller
than the sum of norms of the first $M-1$ terms.
Moreover, (\ref{pjsm}) indicates that each of the first $M-1$ terms
is of $O(1/\sqrt{j}),\ j=1,2,\ldots,M-1$.
Noting this remarkable decaying tendency, we deduce that
the norm of the first term in (\ref{expansion1}) is generally of the same order
as $\|E_m(f)\|$ for all sufficiently smooth functions.
\end{rem}

\section{Upper bounds for $\norm{E_m(e^z,\tau)}$} \label{sec:4}

Here and hereafter let $\beta=\norm{v}$. In this section, we first establish
some new upper bounds for the norm of $E_m(e^z,\tau) = e^{-\tau A}v-\beta
V_me^{-\tau H_m}e_1$. Then we prove theoretically why the first term
in (\ref{expansion1}) generally  measures the error reliably. This justifies
the observation that the first term is numerically ``surprisingly sharp"
in \cite{arnoldi2}, and for the first time provides a solid theoretical support
on the rationale of the error estimates advanced in \cite{arnoldi2}.

Let $\mu_2[A]$ denote the 2-logarithmic norm of $A$, which is defined by
\begin{equation*}\label{lognorm}
  \mu_2[A]=\lim_{h\rightarrow0+}\frac{\norm{I+hA}-1}{h}.
\end{equation*}
The logarithmic norm has plenty of properties; see \cite[p. 31]{book1} and
\cite{other3}. Here we list some of them that will be used later.
\begin{prop}\label{prop4}
Let $\mu_2[A]$ denote the 2-logarithmic norm of $A$. Then we have
\begin{enumerate}[1)]
  \item $-\norm{A}\leq\mu_2[A]\leq\norm{A}$;
  \item $\mu_2[A]=\lambda_{\max}\left(\dfrac{A+A^*}{2}\right)$;
  \item $\mu_2[tA]=t\mu_2[A],\ \text{for all}\ t\geq0$;
  \item $\norm{e^{tA}}\leq e^{t\mu_2[A]},\ \text{for all}\ t\geq 0$.
\end{enumerate}
\end{prop}

\subsection{An upper bound for $\norm{E_m(e^z,\tau)}$} \label{subsec:4.1}

Next we establish an upper bound for the error norm for a general $A$ and refine it
when $A$ is Hermitian.

\begin{theorem}\label{thm4}
With the notation described previously, it holds that
\begin{equation}\label{result1}
\norm{E_m(e^z,\tau)}\leq \tau\beta h_{m+1,m}\max_{0\leq t\leq\tau}
\abs{e_m^Te^{-tH_m}e_1}\dfrac{e^{\tau\mu_2}-1}{\tau\mu_2},
\end{equation}
where $\mu_2=\mu_2[-A]$. \footnote{When revising the paper, we found that,
essentially, (\ref{result1}) is exactly the same as Lemma 4.1 of \cite{botchev},
but the proofs are different and our result is more explicit.
We thank the referee who asked us to compare these
two seemingly different bounds.}
\end{theorem}
\begin{proof}
Let $w(t)=e^{-tA}v$ and $w_m(t)=\beta V_me^{-tH_m}e_1$ be the $m$th Arnoldi
approximation to $w(t)$. Then $w(t)$ and $w_m(t)$ satisfy
\begin{equation}\label{w}
w'(t)=-Aw(t),\ w(0)=v
\end{equation}
and
\begin{equation*}
w_m'(t)=-\beta V_mH_me^{-tH_m}e_1,\ w_m(0)=v,
\end{equation*}
respectively. Using (\ref{Arnoldi decomposition}), we have
\begin{eqnarray}\label{w_m}
w_m'(t) & = & -\beta(AV_m-h_{m+1,m}v_{m+1}e_m^T)e^{-tH_m}e_1 \nonumber\\
& = & -\beta AV_me^{-tH_m}e_1+\beta h_{m+1,m}(e_m^Te^{-tH_m}e_1)v_{m+1} \nonumber\\
& = & -Aw_m(t)+\beta h_{m+1,m}(e_m^Te^{-tH_m}e_1)v_{m+1}.
\end{eqnarray}
Define the error $E_m(e^z,t)=w(t)-w_m(t).$ Then by (\ref{w}) and (\ref{w_m}),
$E_m(e^z,t)$ satisfies
\begin{equation*}
E_m'(e^z,t)=-AE_m(e^z,t)-g(t),\ E_m(e^z,0)=0,
\end{equation*}
where $g(t)=\beta h_{m+1,m}(e_m^Te^{-tH_m}e_1)v_{m+1}$.

By solving the above ODE, we get
\begin{eqnarray}\label{err1}
E_m(e^z,\tau) & = & -\int_0^\tau e^{(t-\tau)A}g(t)dt\nonumber\\
& = & -\beta h_{m+1,m}\int_0^\tau
(e_m^Te^{-tH_m}e_1)e^{(t-\tau)A}v_{m+1}dt.\nonumber
\end{eqnarray}
Taking the norms on the two sides gives
\begin{equation*}
\norm{E_m(e^z,\tau)}\leq \tau\beta h_{m+1,m}\max_{0\leq t\leq\tau}
\abs{e_m^Te^{-tH_m}e_1}\dfrac{e^{\tau\mu_2}-1}{\tau\mu_2},
\end{equation*}
where $\mu_2=\mu_2[-A]$.
\end{proof}


In the case that $A$ is Hermitian, the Arnoldi decomposition reduces to the
Lanczos decomposition
\begin{equation}\label{lanczos}
AV_m=V_mT_m+\eta_{m+1}v_{m+1}e_m^T,
\end{equation}
where $T_m$ is Hermitian tridiagonal. The $m$th Lanczos approximation to
$e^{-\tau A}v$ is $\beta V_me^{-\tau T_m}e_1$.
For $A$ Hermitian, since $\mu_2=\mu_2[-A]=\lambda_{\max}\left(-\frac{A+A^*}{2}\right)
=-\lambda_{\min}(A)$,
it follows from (\ref{result1}) that
\begin{equation}
    \norm{E_m(e^z,\tau)}\leq \tau\beta \eta_{m+1}\max_{0\leq t\leq\tau}
    \abs{e_m^Te^{-tT_m}e_1}\dfrac{e^{-\tau\lambda_{\min}(A)}-1}{-\tau
    \lambda_{\min}(A)}, \label{lanczoserror}
\end{equation}
which coincides with Theorem 3.1 in \cite{error8} by setting
$\alpha=0$ or $\tau$ there. We remark that, by the Taylor
expansion, it is easily justified that the factor
$$
\dfrac{e^{-\tau\lambda_{\min}(A)}-1}{-\tau\lambda_{\min}(A)}\geq 1
$$
if $A$ is semi-negative definite and
$$
\dfrac{e^{-\tau\lambda_{\min}(A)}-1}{-\tau\lambda_{\min}(A)}\leq 1
$$
if $A$ is semi-positive definite.

From the above theorem, for $A$ Hermitian we can derive a
more compact form, which is practically more convenient to use.

\begin{theorem}\label{thm1}
Assume that $spec(A)$ is contained in the interval $\Lambda\equiv [a,\,b]$ and the
Lanczos process (\ref{lanczos}) can be run $m$ steps without breakdown. Then
\begin{equation}\label{result3}
    \norm{E_m(e^z,\tau)}\leq \gamma_1\tau\beta\eta_{m+1}\abs{e_m^Te^{-\tau T_m}e_1},
\end{equation}
where
\begin{gather}\label{gamma1}
  \gamma_1=\left\{
                   \begin{array}{ll}
                     e^{\tau(b-a)}\dfrac{e^{-\tau\lambda_{\min}(A)}-1}{-\tau
                     \lambda_{\min}(A)}, & \lambda_{\min}(A)\neq0 \\
                     e^{\tau(b-a)}, &  \lambda_{\min}(A)=0. \\
                   \end{array}
                   \right.
\end{gather}
\end{theorem}

\begin{proof}
As before, the error $E_m(e^z,\tau)$ satisfies
\begin{equation}\label{err2}
E_m(e^z,\tau)=-\beta \eta_{m+1}\int_0^\tau
(e_m^Te^{-tT_m}e_1)e^{(t-\tau)A}v_{m+1}dt.
\end{equation}
Let $\lambda_1,\ldots,\lambda_m$ be the $m$ eigenvalues of $T_m$ and $f(z)=e^z$.
According to Propositions \ref{prop1}--\ref{prop2}, it is easy to verify
that for all $t\geq0$,
$$e^{-tT_m}\ =\ \sum_{i=0}^{m-1}a_i(t)(-tT_m)^i,$$
where $a_{m-1}(t)=f[-t\lambda_1,\ldots, -t\lambda_m],\ t\geq 0.$

From the tridiagonal structure of $T_m$, we have
\begin{equation}\label{err-t}
e_m^Te^{-tT_m}e_1=(-t)^{m-1}a_{m-1}(t)e_m^TT_m^{m-1}e_1.
\end{equation}
Combining (\ref{err2}) and (\ref{err-t}), we get
\begin{equation}\label{err3}
E_m(e^z,\tau)\ =\ -\beta \eta_{m+1}e_m^Te^{-\tau T_m}e_1\int_{0}^{\tau}%
\left(\frac{t}{\tau}\right)^{m-1}\frac{a_{m-1}(t)}{a_{m-1}(\tau)}e^{(t-\tau)A}v_{m+1}dt.
\end{equation}
Denote $-t\Lambda\equiv[-tb,\,-ta],\ 0\leq t\leq\tau.$ (\ref{diff1}) shows that
there exist $\zeta_1\in -t\Lambda$ and $\zeta_2\in -\tau \Lambda$, such that
\begin{eqnarray}\label{inequality2}
  \abs{\frac{a_{m-1}(t)}{a_{m-1}(\tau)}} &=&
  \abs{\frac{f^{(m-1)}(\zeta_1)}{f^{(m-1)}(\zeta_2)}}
    \leq\frac{\max_{\zeta\in -t\Lambda}\abs{f^{(m-1)}(\zeta)}}
    {\min_{\zeta\in -\tau \Lambda}\abs{f^{(m-1)}(\zeta)}}\nonumber \\
    &\leq & \frac{\max_{\zeta\in -\tau \Lambda}\abs{f^{(m-1)}(\zeta)}}
    {\min_{\zeta\in -\tau \Lambda}\abs{f^{(m-1)}(\zeta)}}
    =e^{\tau(b-a)}.
\end{eqnarray}
From the above and (\ref{err3}), we get
\begin{eqnarray*}
\norm{E_m(e^z,\tau)} & \leq &\beta\eta_{m+1}\abs{e_m^Te^{-\tau T_m}e_1}\int_{0}^{\tau}%
\left(\frac{t}{\tau}\right)^{m-1}e^{\tau(b-a)}\norm{e^{(t-\tau)A}}dt\\
& \leq &\beta\eta_{m+1}\abs{e_m^Te^{-\tau T_m}e_1}e^{\tau(b-a)}\int_{0}^{\tau}
\left(\frac{t}{\tau}\right)^{m-1}e^{(\tau-t)\mu_2}dt\\
& \leq &\gamma_1\tau\beta\eta_{m+1}\abs{e_m^Te^{-\tau T_m}e_1},
\end{eqnarray*}
where $\gamma_1$ is defined as (\ref{gamma1}) and $\mu_2=\mu_2[-A]$.
\end{proof}

Consider the linear system $Ax=v$. It is known \cite[p. 159-160]{book6} that
the Lanczos approximation $x_m$ to $x=A^{-1}v$ is given by $x_m=\beta V_mT_m^{-1}e_1$
and the a posteriori residual $r_m=v-Ax_m$ satisfies
$$
r_m=\beta\eta_{m+1}(e_m^TT_m^{-1}e_1)v_{m+1},
$$
whose norm is $\|r_m\|=\beta\eta_{m+1}\abs{e_m^TT_m^{-1}e_1}$.
The the error $e_m=x-x_m$ is closely related to $r_m$ by
$e_m=A^{-1}r_m$. Therefore, we have
\begin{equation}\label{EAx}
    \norm{e_m}=\norm{A^{-1}r_m}\leq \|A^{-1}\|\|r_m\|=
    \tilde{\gamma}_1\beta\eta_{m+1}\abs{e_m^TT_m^{-1}e_1},
\end{equation}
where $\tilde{\gamma}_1=\norm{A^{-1}}$.

In the same spirit, Theorem \ref{thm1} gives a similar result for the Lanczos
approximation to $f(A)v$, where the left-hand side of
the error norm (\ref{result3}) is uncomputable in practice, while its
right-hand side excluding the factor $\gamma_1$ is computable
and can be interpreted as an a posteriori
error. We can write (\ref{result3}) and (\ref{EAx}) in a unified form
\begin{equation*}
    \norm{f(A)v-\beta V_mf(T_m)e_1}\leq\gamma\beta\eta_{m+1}\abs{e_m^Tf(T_m)e_1},
\end{equation*}
where $\gamma$ is a constant depending on the spectrum of $A$ and
$f(z)=z^{-1}$ or $e^z$.

\subsection{A second upper bound for $\norm{E_m(e^z,\tau)}$} \label{subsec:4.2}

We now analyze expansion (\ref{expansion1}) when $f(z)=e^{z}$ and
the sequence $\{\phi_k(z)\}$ are defined as (\ref{phi}), derive
compact upper bounds for the first term and the sum of the rest in expansion
(\ref{expansion1}), and for the first time prove that $\norm{E_m(e^z,\tau)}$
is determined by the first term of the error expansion. This
is one of our main results in this paper.

For the Arnoldi approximation (\ref{Arnoldi approximation}) to $e^{-tA}v$,
since $F(H_m)\subseteq F(A)$ for $1\leq m\leq N$, in this subsection we
take the set $\Omega$ to be a closed
convex set containing the field of values $F(A)$ in
Theorem~\ref{error expansion1}. As a result,
expansion (\ref{expansion1}) becomes
$$E_m(e^z,t)=-t\beta h_{m+1,m}\sum_{k=1}^{\infty}%
(-t)^{k-1}e_m^T\phi_k(-t H_m)e_1 q_{k-1}(A)v_{m+1},$$
where $q_0=1,q_k=(z-z_0)\cdots(z-z_{k-1}),\ z_i\in F(A),\ i=0,\ldots,k-1,\ k\geq1.$

Denoting
$$E_m^{(2)}(e^z,t)=-t\beta h_{m+1,m}\sum_{k=2}^{\infty}%
(-t)^{k-1}e_m^T\phi_k(-t H_m)e_1 q_{k-1}(A)v_{m+1},
$$
we can present the following results.

\begin{theorem}\label{thm2}
Let $\mu_2=\mu_2[-A]$. Then
$E_m^{(2)}(e^z,\tau)$ and $E_m(e^z,\tau)$ satisfy
\begin{eqnarray}\label{result2}
\norm{E_m^{(2)}(e^z,\tau)}&\leq&\gamma_2\tau\beta h_{m+1,m}
\max_{0\leq t\leq\tau}\abs{e^T_m\phi_1(-tH_m)e_1},\nonumber\\
\norm{E_m(e^z,\tau)}&\leq& (1+\gamma_2)\tau\beta h_{m+1,m} \max_{0\leq t
\leq\tau}\abs{e^T_m\phi_1(-tH_m)e_1},
\end{eqnarray}
respectively, where
\begin{gather}\label{gamma2}
  \gamma_2=\left\{
  \begin{array}{ll}
  \norm{q_1(A)v_{m+1}}\dfrac{e^{\tau\mu_2}-1}{\mu_2}, & \mu_2\neq0 \\
  \tau \norm{q_1(A)v_{m+1}}, &  \mu_2=0.\\
  \end{array}
  \right.
\end{gather}
\end{theorem}

\begin{proof}
Define the $(m+1)\times(m+1)$ matrix
\begin{equation*}
\overline{H}_m\equiv\left(\begin{array}{cc}H_m-z_0I & \mathbf{0} \\
h_{m+1,m}e_m^T & 0\end{array}\right)
\end{equation*} and
$$
w(t)=e^{-tA}v,\ w_m^{(2)}(t)=\beta V_{m+1}e^{-t(\overline{H}_m+z_0I)}e_1,
$$
where $H_m,\,h_{m+1,m}$ and $V_{m+1}=[V_m,\,v_{m+1}]$ are generated
by the Arnoldi decomposition (\ref{Arnoldi decomposition}). Then
\begin{equation}\label{e_tT}
e^{-t\overline{H}_m}=\left(\begin{array}{ccc}e^{tz_0}e^{-tH_m} & & \mathbf{0} \\
-th_{m+1,m}e^{tz_0}e_m^T\phi_1(-tH_m) & & 1\end{array}\right)
\end{equation} and
\begin{eqnarray}\label{em2}
  w(t)- w_m^{(2)}(t) &=& e^{-tA}v-\beta V_{m+1} \left(\begin{array}{c}
                                                        e^{-tH_m}e_1\\
                                                        -th_{m+1,m}e_m^T\phi_1(-tH_m)e_1
                                                        \end{array}\right)\nonumber\\
   &=& e^{-tA}v-\beta V_me^{-tH_m}+t\beta h_{m+1,m}e_m^T\phi_1(-tH_m)e_1 v_{m+1}\\
   &=& E_m^{(2)}(e^z,t)\nonumber.
\end{eqnarray}
According to
\begin{align*}
w'(t) &= -Aw(t), & w(0) &= v,\\
w_m^{(2)'}(t) & = -\beta V_{m+1}(\overline{H}_m+z_0I)e^{-t(\overline{H}_m+z_0I)}
e_1, & w_m^{(2)}(0) &= v,
\end{align*}
we get
$$
E_m^{(2)'}(e^z,t)=-Aw(t)+\beta V_{m+1}(\overline{H}_m+z_0I)
e^{-t(\overline{H}_m+z_0I)}e_1.
$$
From (\ref{Arnoldi decomposition}), we have
$$V_{m+1}(\overline{H}_m+z_0I)=[\ V_m,\ v_{m+1}\ ]\left[\begin{array}{cc}H_m &
\mathbf{0} \\
h_{m+1,m}e_m^T & z_0\end{array}\right]=[\ AV_m,\ z_0v_{m+1}\ ].$$
 Then it follows from the above that
\begin{eqnarray*}
E_m^{(2)'}(e^z,t) & = & -Aw(t)+\beta[\ AV_m,\ z_0v_{m+1}\ ]
e^{-t(\overline{H}_m+z_0I)}e_1\\
& = & -A\left(w(t)-w_m^{(2)}(t)\right)-Aw_m^{(2)}(t)+
\beta[\ AV_m,\ z_0v_{m+1}\ ]e^{-t(\overline{H}_m+z_0I)}e_1\\
& = & -AE_m^{(2)}(e^z,t)-\beta [\ \mathbf{0},\ q_1(A)v_{m+1}\ ]
e^{-t(\overline{H}_m+z_0I)}e_1,
\end{eqnarray*}
where $q_1(z)=z-z_0$. Therefore, from (\ref{e_tT}) we get
\begin{equation*}
\Big\{ \begin{array}{l}E_m^{(2)'}(e^z,t)=-AE_m^{(2)}(e^z,t)+g(t)\\
E_m^{(2)}(e^z,0)=0,
\end{array}
\end{equation*}
where $g(t)=t\beta h_{m+1,m}e^T_m\phi_1(-tH_m)e_1 q_1(A)v_{m+1}$.

Solving the above ODE for $E_m^{(2)}(e^z,\tau)$, we obtain
\begin{equation}\label{e_m2}
E_m^{(2)}(e^z,\tau)=\beta h_{m+1,m}\left(\int_0^{\tau}te^T_m\phi_1(-tH_m)e_1
\ e^{(t-\tau)A}dt\right)\ q_1(A)v_{m+1}.
\end{equation}
Taking the norms on the two sides and defining $\gamma_2$ as (\ref{gamma2}),
we get
\begin{equation*}
\norm{E_m^{(2)}(e^z,\tau)}\leq\gamma_2\tau\beta h_{m+1,m}
\max_{0\leq t\leq\tau}\abs{e^T_m\phi_1(-tH_m)e_1}
\end{equation*}
and
\begin{eqnarray*}
 \norm{E_m(e^z,\tau)} &\leq& \tau\beta h_{m+1,m}\abs{e^T_m\phi_1
 (-\tau H_m)e_1}+\norm{E_m^{(2)}(e^z,\tau)} \\
    &\leq& (1+\gamma_2)\tau\beta h_{m+1,m} \max_{0\leq t\leq\tau}
    \abs{e^T_m\phi_1(-tH_m)e_1},
\end{eqnarray*}
which completes the proof.
\end{proof}

Let $f(z)=e^z$ and $\phi_1(z)=\dfrac{e^z-e^{z_0}}{z-z_0}$. Then
the divided differences of $f(z)$ and $\phi_1(z)$ have the following relationship.
\begin{lem}\label{lem1}
Given $z_1,z_2,\ldots,z_m \in \Omega$, we have
\begin{equation}\label{diff2}
\phi_1[z_1,\ldots,z_m]=f[z_0,z_1,\ldots,z_m].
\end{equation}
\end{lem}
\begin{proof}
We prove the lemma by induction. For $i=1$,
$$\phi_1[z_1]=\frac{e^{z_1}-e^{z_0}}{z_1-z_0}=f[z_0,z_1].$$
So (\ref{diff2}) is true.

Assume that (\ref{diff2}) holds for $i=k,\ k<m$, i.e.,
$$\phi_1[z_1,\ldots,z_k]=f[z_0,z_1,\ldots,z_k].$$
Then for $i=k+1$, we get
\begin{eqnarray*}
\phi_1[z_1,\ldots,z_{k+1}] & = & \frac{\phi_1[z_1,\ldots,z_{k-1},z_{k+1}]-
\phi_1[z_1,\ldots,z_k]}{z_{k+1}-z_k}\\
& = & \frac{f[z_0,z_1,\ldots,z_{k-1},z_{k+1}]-f[z_0,z_1,\ldots,z_k]}{z_{k+1}-z_k}\\
& = & f[z_0,z_1,\ldots,z_{k+1}].
\end{eqnarray*}
Thus, the lemma is true.
\end{proof}

With $\gamma_1$ and $\mu_2$ defined as in Theorem \ref{thm1}, we can refine
Theorem~\ref{thm2} and get an explicit and compact bound for
$\norm{E_m(e^z,\tau)}$ when $A$ is Hermitian.

\begin{theorem}\label{thm3}
Assume that $spec(A)$ is contained in the interval $\Lambda\equiv[a,\,b]$ and
the Lanczos process {\rm (\ref{lanczos})} can be run $m$ steps without breakdown.
Then we have
\begin{eqnarray*}
\norm{E_m^{(2)}(e^z,\tau)}&\leq& \gamma_3\tau\beta\eta_{m+1}
\abs{e_m^T\phi_1(-\tau T_m)e_1},\\
\norm{E_m(e^z,\tau)}&\leq &(1+\gamma_3)\tau\beta\eta_{m+1}
\abs{e_m^T\phi_1(-\tau T_m)e_1},
\end{eqnarray*}
where $\gamma_3=\tau\gamma_1\norm{(A-z_0I)v_{m+1}}$ and
$\gamma_1$ is defined as {\rm (\ref{gamma1})} for any $z_0\in \Lambda$.
\end{theorem}

\begin{proof}
It follows from (\ref{e_m2}) that for a Hermitian $A$ we have
\begin{equation*}\label{myerror}
E_m^{(2)}(e^z,\tau)=\beta \eta_{m+1}\left(\int_0^\tau te^T_m\phi_1(-tT_m)e_1\
e^{(t-\tau)A}dt\right)(A-z_0I)v_{m+1}.
\end{equation*}
Let $\lambda_1,\ldots,\lambda_m$ be the eigenvalues of $T_m$ and $z_0\in \Lambda$.
By Propositions \ref{prop1}--\ref{prop2} and Lemma \ref{lem1}, we know that
for all $t\geq0$
$$
\phi_1(-tT_m)=\sum_{i=0}^{m-1}\hat{a}_i(t)(-tT_m)^i,
$$
where
\begin{eqnarray*}
\hat{a}_{m-1}(t) & = & \phi_1[-t\lambda_1,\ldots, -t\lambda_m]\\
& = & f[-tz_0,-t\lambda_1,\ldots, -t\lambda_m].
\end{eqnarray*}
Similar to the proof of Theorem \ref{thm1}, we get
\begin{equation}\label{myerr3}
E_m^{(2)}(e^z,\tau)=\beta \eta_{m+1}e_m^T\phi_1(-\tau T_m)e_1\left(\int_{0}^{\tau}%
\frac{t^m}{\tau^{m-1}}\frac{\hat{a}_{m-1}(t)}{\hat{a}_{m-1}(\tau)}
e^{(t-\tau)A}dt\right)(A-z_0I)v_{m+1}
\end{equation}
and
\begin{equation}\label{inequality1}
   \abs{\frac{\hat{a}_{m-1}(t)}{\hat{a}_{m-1}(\tau)}}\leq
   \frac{\max_{\zeta\in -t\Lambda}\abs{f^{(m)}(\zeta)}}
   {\min_{\zeta\in -\tau \Lambda}\abs{f^{(m)}(\zeta)}}\leq
   \frac{\max_{\zeta\in -\tau \Lambda}\abs{f^{(m)}(\zeta)}}{\min_{\zeta\in -\tau
   \Lambda}\abs{f^{(m)}(\zeta)}}=e^{\tau(b-a)}.
\end{equation}
Substituting it into (\ref{myerr3}) gives
\begin{eqnarray*}
\norm{E_m^{(2)}(e^z,\tau)} 
& \leq &\tau\beta\eta_{m+1}\abs{e_m^T\phi_1(-\tau T_m)e_1}e^{\tau(b-a)}
\norm{(A-z_0 I)v_{m+1}}\int_{0}^{\tau}e^{(\tau-t)\mu_2}dt\\
& \leq &\gamma_3\tau\beta\eta_{m+1}\abs{e_m^T\phi_1(-\tau T_m)e_1}
\end{eqnarray*}
and
\begin{eqnarray*}
\norm{E_m(e^z,\tau)} 
\int_{0}^{\tau}
& \leq &\tau\beta\eta_{m+1}\abs{e_m^T\phi_1(-\tau T_m)e_1}+\norm{E_m^{(2)}
(e^z,\tau)}\\
& \leq &(1+\gamma_3)\tau\beta\eta_{m+1}\abs{e_m^T\phi_1(-\tau T_m)e_1},
\end{eqnarray*}
where $\gamma_3=\tau\gamma_1\norm{(A-z_0 I)v_{m+1}}$.
\end{proof}

\begin{rem}
Remarkably, noting that $\|(A- z_0 I)v_{m+1}\|$ is
typically comparable to $\|A\|$ whenever $z_0\in F(A)$,
Theorem~\ref{thm3} shows that the error norm
$\norm{E_m(e^z,\tau)}$ is essentially determined by the first term
of the error expansion provided that
$\gamma_3$, or equivalently $\gamma_1$ is mildly sized.

\end{rem}

\section{Some a posteriori error estimates for approximating $f(A)v$}
\label{sec:5}

Previously we have established the error expansion of the Krylov-like approximation
for sufficiently smooth functions $f(z)$ and derived some upper bounds for the
Arnoldi approximation to $e^{-\tau A}v$. They form the basis of seeking
reliable a posteriori error estimates.

Define
\begin{equation} \label{criter}
    \xi_1=\beta h_{m+1,m}\abs{e_m^Tf(H_m)e_1}\ \text{and}\ \xi_2=
    \beta h_{m+1,m}\abs{e_m^T\phi_1(H_m)e_1}.
\end{equation}
Formally, as reminiscent of the residual formula for the Arnoldi
method for solving linear systems, Saad \cite{arnoldi2} first
proposed using $\xi_1$ as an a posteriori
error estimate when $f(z)=e^z$ without any theoretical
support or justification. Due to the lack of
a definition of residual, an insightful interpretation of
$\xi_1$ as an a posteriori estimate is not always immediate.
In \cite{rd4,other5,polykry4}, the authors have
introduced a generalized residual notion for the Arnoldi
approximation to $f(A)v$, which can be used to justify
the rationale of $\xi_1$. Let us briefly review why it is so.
By the Cauchy integral definition \cite[p. 8]{book3}, the error
of the $m$th Arnoldi approximation to $f(A)v$ can be expressed as

\begin{equation}\label{ECauchy}
    E_m(f)=\frac{1}{2\pi i}\int_\Gamma f(\lambda)\left[(\lambda I-A)^{-1}v-
    \beta V_m(\lambda I-H_m)^{-1}e_1\right]d\lambda,
\end{equation}
where $\Gamma$ is the contour enclosing $F(A)$. For the Arnoldi method
for solving the linear system $(\lambda I-A)x=v$, the error is
\begin{equation*}
    e_m(\lambda)=(\lambda I-A)^{-1}v-\beta V_m(\lambda I-H_m)^{-1}e_1,
\end{equation*}
and the residual is $r_m(\lambda)=(\lambda I-A)e_m(\lambda)$,
which, by the Arnoldi decomposition, is expressed as
\begin{equation}\label{rm1}
    r_m(\lambda)=\beta h_{m+1,m}\left(e_m^T(\lambda I-H_m)^{-1}e_1\right)v_{m+1}.
\end{equation}
From the notations above, (\ref{ECauchy}) becomes
\begin{equation*}
    E_m(f)=\frac{1}{2\pi i}\int_\Gamma f(\lambda)e_m(\lambda)d\lambda.
\end{equation*}
Replacing $e_m(\lambda)$ by $r_m(\lambda)$, one defines the generalized residual
\begin{equation*}
    R_m(f)=\frac{1}{2\pi i}\int_\Gamma f(\lambda)r_m(\lambda)d\lambda,
\end{equation*}
which exactly equals the standard residual $r_m$ of the linear system
$(\lambda I-A)x=v$ by taking $f(\lambda)=(\lambda I-A)^{-1}$.
Substituting (\ref{rm1}) into the above leads to
\begin{equation*}\label{rm2}
    R_m(f)=\beta h_{m+1,m}\left(e_m^Tf(H_m)e_1\right)v_{m+1}.
\end{equation*}
Since $\xi_1=\norm{R_m(f)}$, $\xi_1$ can be interpreted as an a posteriori
error estimate and used to design a stopping criterion
for the Arnoldi approximation to $f(A)v$ for an analytic function $f(z)$ over
$\Gamma$.

The theoretical validity of $\xi_2$ as an a posteriori error
estimate has been unclear and not been justified even for $f(z)=e^z$ because
there has been no estimate for
$\|E_m^{(2)}(e^z,\tau)\|$ before. With our bounds established in Section 4,
it is now expected that both $\xi_1$ and $\xi_2$ can be naturally
used as reliable a posteriori error estimates for $\norm{E_m(e^z,\tau)}$
{\em without the help of the generalized residual notion}. As a matter of fact,
our bounds have essentially given close relationships between the a
priori error norm $\|E_m(e^z,\tau)\|$ and the a posteriori
quantities $\xi_1$ and $\xi_2$. Furthermore, based on the error expansion
in Theorem \ref{error expansion1} and the remarks followed,
we can justify the rationale of $\xi_2$ as an a posteriori error estimate for
the error norm $\norm{E_m(f)}$ of the Krylov-like approximation for
sufficiently smooth functions $f(z)$.
In what follows we will give more details,
showing why $\xi_1$ and $\xi_2$ are generally reliable and accurate
a posteriori error estimates. We will confirm their effectiveness by
numerical experiments.

\subsection{The case that $A$ is Hermitian} \label{subsec:5.1}

For a Hermitian $A$, we can give a new justification of $\xi_1$
for the Arnoldi approximation to $e^{-\tau A}v$. Theorem \ref{thm1} shows that
$\norm{E_m(e^z,\tau)}\leq \gamma_1\xi_1$, where $\gamma_1$ is defined
as (\ref{gamma1}), while Theorem \ref{thm3} indicates
that $\norm{E_m(e^z,\tau)}\leq (1+\gamma_3)\xi_2$ with $\gamma_3=
\tau\gamma_1\norm{(A-z_0I)v_{m+1}}$. Note that whenever $\gamma_1$
is mildly sized, that both $\xi_1$ and $\xi_2$ must be
very good estimates for the true error norm. Note that inequalities
(\ref{inequality2}) and (\ref{inequality1}) should generally be conservative
as the factor $\gamma_1$ is maximized in the worst case. Thus, each of
$\xi_1$ and $\xi_2$ is expected to determine the
error norm $\norm{E_m(e^z,\tau)}$ reliably
in practical computations, even though $\gamma_1$ or $\gamma_3$ is large.
Numerical experiments will illustrate that our relative estimates
defined by (\ref{relest})  mimic the relative error defined by
(\ref{relerror}) very well for very big $\gamma_1$ and $\gamma_3$.

All the experiments in this paper are performed on Intel(R) Core(TM)2 Duo
CPU T6600 @2.20GHz with RAM 2.00 GB using {\sc Matlab} 7.8.0 under
a Windows XP operating system. $f(H_m)$ and $\phi_1(H_m)$ are computed by means of
the spectral decomposition when $H_m$ is Hermitian or by the {\sc Matlab} built-in
functions {\bf funm}, which is replaced by \textbf{expm} if $f(z)=e^z$.
To illustrate the effectiveness of $\xi_1$ and $\xi_2$, we compute the
``exact" solution $f(A)v$ by first using the above functions to calculate
$f(A)$ explicitly and then multiplying it with $v$. Keep in mind that $\Omega$
contains the field of values $F(H_m)$ and we can
always require all $z_0,z_1,\ldots \in F(H_m)$. However, it is only
$z_0$ that is needed to define $\phi_1(z)$ in order to compute $\xi_2$.
Note that all the diagonal entries $h_{i,i}$ of $H_m$ lie in $F(H_m)$.
Therefore, to be unique, here and hereafter, in practical implementations,
we simply take
$$
z_0=h_{1,1}
$$
without any extra cost. Certainly,
there are infinitely many choices of $z_0$. Experimentally, we have found that
the choice of $z_0$ has little essential effect on the size of $\xi_2$.
With $z_0$ given, we have the function $\phi_1(z)$ and can compute
the approximation $f_m=\beta V_mf(H_m)e_1$. We compute
the a posteriori error estimates $\xi_1$ and $\xi_2$
using the method in \cite{arnoldi2,Sidje98}. Let $\bar{H}_m=\left(
                                                   \begin{array}{cc}
                                                     H_m & \textbf{0} \\
                                                     e_m^T & z_0 \\
                                                   \end{array}
                                                 \right).
$ It is known that if $f(z)$ is analytic on and inside $\Omega$ then
$$f(\bar{H}_m)=\left(
               \begin{array}{cc}
                 f(H_m) & \textbf{0} \\
                 e_m^T\phi_1(H_m) & f(z_0) \\
               \end{array}
             \right).
$$
As a result, $f_m=\beta V_m\left[f(\bar{H}_m)e_1\right]_{1:m},$
$\xi_1=\beta \abs{\left[f(\bar{H}_m)e_1\right]_{m}},$ and
$\xi_2=\beta \abs{\left[f(\bar{H}_m)e_1\right]_{m+1}}.$

The Arnoldi decomposition is performed with the modified Gram-Schmidt process
(see, e.g., \cite{arnoldi2}) until the approximation
$f_m=\beta V_mf(H_m)e_1$ satisfies
\begin{equation} \label{relerror}
\norm{f(A)v-f_m}/\norm{f(A)v}\leq \epsilon.
\end{equation}
We take $\epsilon=10^{-12}$ in the experiments. We define relative
posterior estimates as
\begin{equation} \label{relest}
\xi_1^{rel}:=\frac{\xi_1}{\|f_m\|}, \ \xi_2^{rel}:=\frac{\xi_2}{\|f_m\|},
\end{equation}
and compare them with the true relative error (\ref{relerror}). We
remark that in (\ref{relest}) we use the easily computable quantity
$\|f_m\|=\beta \|f(H_m)e_1\|$ to replace $\|f(A)v\|$, which is unavailable
in practice. Note that $\|f_m\|$ approximates $\|f(A)v\|$ whenever $f_m$
approaches $f(A)v$. Therefore, the sizes of $\xi_1^{rel}$  and $\xi_2^{rel}$
are very comparable to their corresponding counterparts that
use $\|f(A)v\|$ as the denominator once the convergence is starting.
If $f_m$ is a poor approximation to $f(A)v$, as is typical in very first
steps, both the error in (\ref{relerror}) and error estimates in
(\ref{relest}) are not small. Nevertheless, such replacement does not cause
any problem  since whether or not accurately estimating
a large error is unimportant, and what is of interest is
to reasonably check the convergence with increasing $m$.


\emph{Example 1.} We justify the effectiveness of $\xi_1^{rel}$ and $\xi_2^{rel}$
for $f(A)=e^{-\tau A}$. Consider the diagonal matrix $A$ of size $N=1001$ taken
from \cite{error2} with equidistantly spaced eigenvalues in the interval
$[0,\,40]$. The $N$-dimensional vector $v$ is generated randomly in
a uniform distribution with $\norm{v}=1$. For this $A$ and $\tau=0.1,\ 0.5, 1$,
we have $\lambda_{\min}(A)=0$ and $\gamma_1=e^4,\ e^{20},\ e^{40}$, which
are approximately $ 54.6,\ 4.9\times 10^8,\
2.4\times 10^{17}$, respectively. We see $\gamma_1$ varies drastically with
increasing $\tau$. The corresponding three $\gamma_3=\tau \gamma_1
\|(A-z_0 I)v_{m+1}\|$ are the same order as $\tau \gamma_1\|A\|$
and change in a similar way.

\begin{figure}[!htb]
\centering
\begin{minipage}[t]{0.45\linewidth}
    \centering
    \includegraphics[width=6.5cm,height=6.5cm]{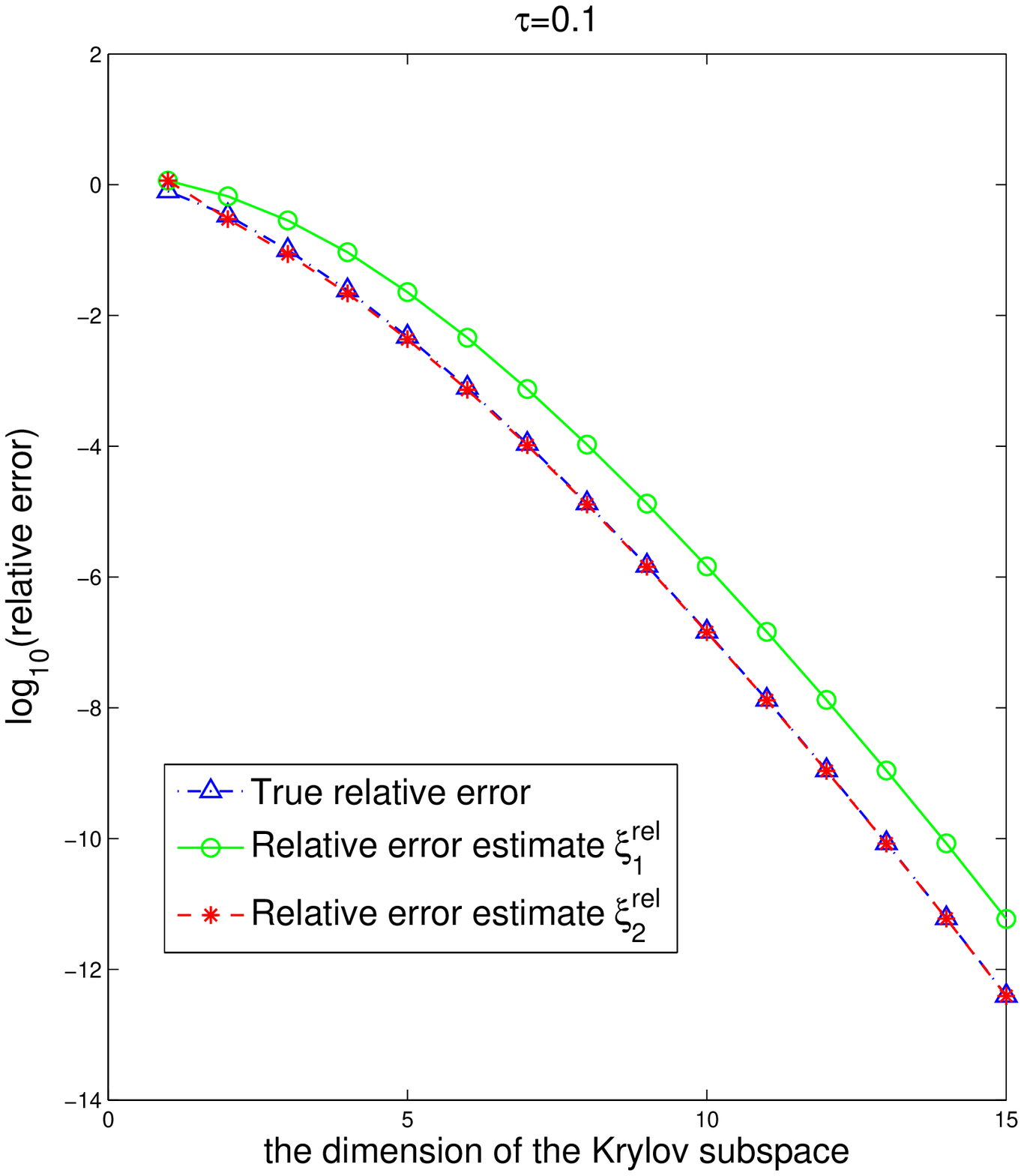}
\end{minipage}
\hspace{2ex}
\begin{minipage}[t]{0.45\linewidth}
    \centering
    \includegraphics[width=6.5cm,height=6.5cm]{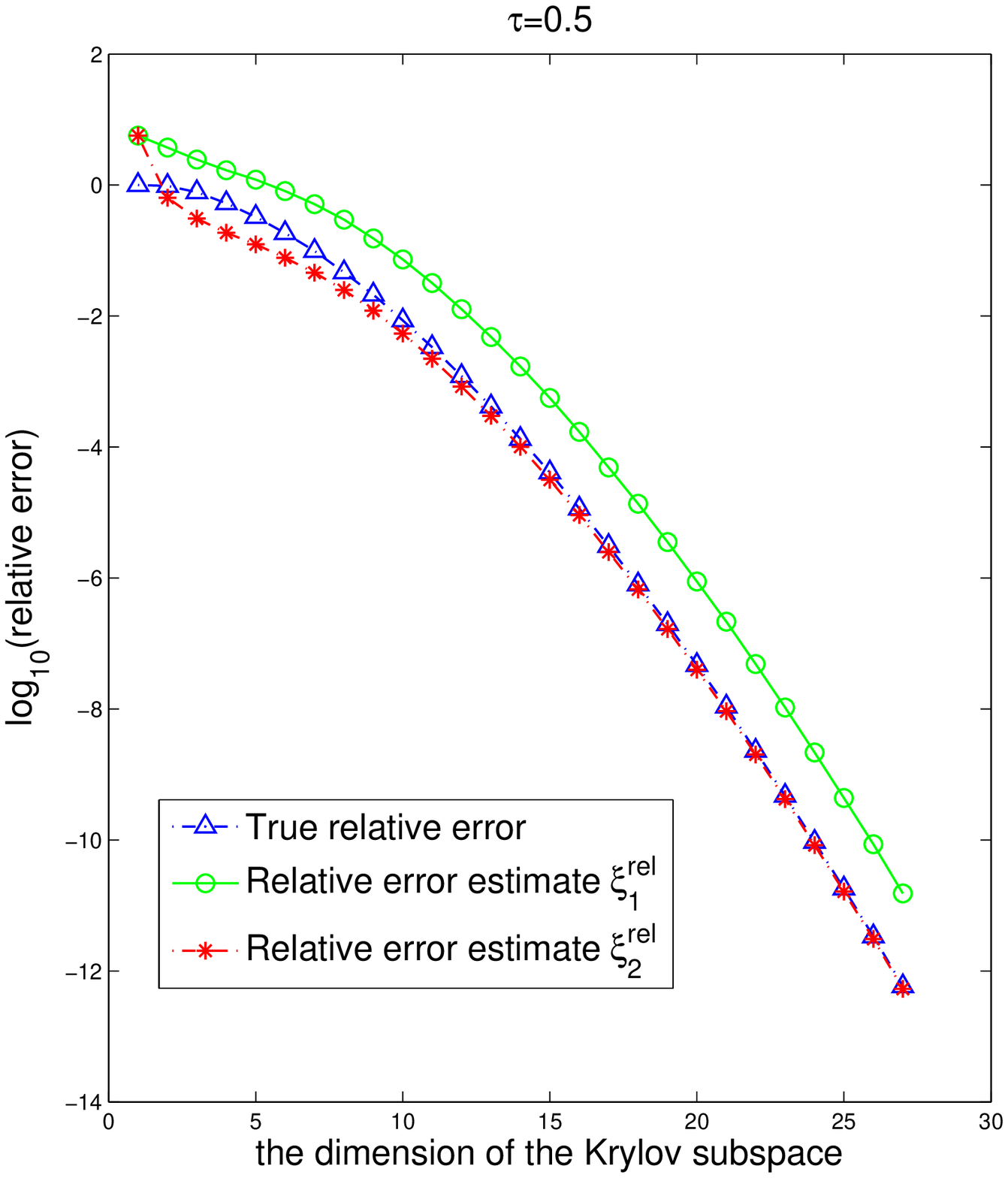}
\end{minipage}
\centering
\includegraphics[width=6.5cm,height=6.5cm]{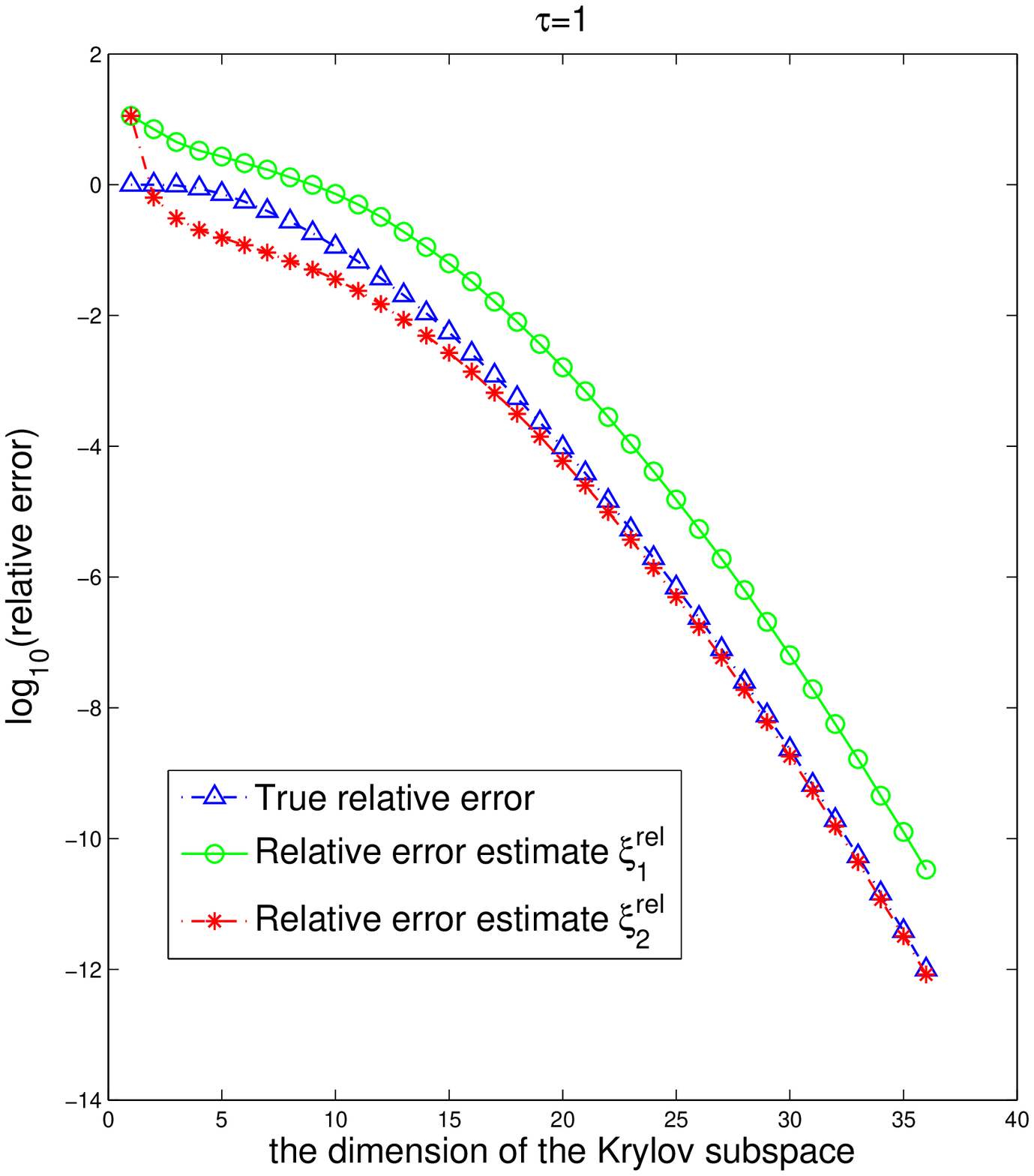}
\caption{Example 1: The relative error estimates and the true relative error
for $e^{-\tau A}v$ for $A$ symmetric with $N=1001$}\label{figure1}
\end{figure}

Figure~\ref{figure1} depicts the curves of the two relative error estimates
$\xi_1^{rel}$, $\xi_2^{rel}$ and the true relative error (\ref{relerror}) for
the Lanczos approximations
to $e^{-\tau A}v$ with three parameters $\tau=0.1,\,0.5$ and $1$.
It is seen that both estimates $\xi_1^{rel}$ and $\xi_2^{rel}$
are accurate, and particularly $\xi_2^{rel}$ is indistinguishable from the
true relative error
as $m$ increases and is sharper than $\xi_1^{rel}$ by roughly one order.
Remarkably, it is observed from the figure that
the effectiveness of $\xi_1^{rel}$ and $\xi_2^{rel}$ is little affected by
greatly varying $\gamma_1$ and $\gamma_3$, which themselves critically
depend on the spectrum of $A$ and $\tau$. This
demonstrates that $\gamma_1$ and $\gamma_3$ are too conservative,
they behave like $O(1)$ in practice and
our estimates are accurate and mimic the true error norms very well.
However, as a byproduct, we find that the smaller $\tau$ is,
the faster the Lanczos approximations converge. So the convergence
itself is affected by the size of $\tau$ considerably.

\subsection{The case that $A$ is non-Hermitian} \label{subsec:5.2}

Unlike the Hermitian case, for $A$ non-Hermitian, Theorems \ref{thm4}
and \ref{thm2} do not give explicit relationships between the error
norm and $\xi_1$, $\xi_2$. So it is not obvious how to use $\xi_1$ and $\xi_2$ to
estimate the true error norm. However, if we lower our standard a little bit, it
is instructive to make use of these two theorems
to justify the effectiveness of $\xi_1^{rel}$ and $\xi_2^{rel}$ for estimating
the relative error of the Arnoldi approximations to $e^{-\tau A}v$,
as shown below.

Define two functions $g_1(t)=e_m^Te^{-tH_m}e_1$ and $g_2(t)=e_m^T\phi_1(-tH_m)e_1$.
By continuity, there exist two positive constants $c_1,\ c_2\geq 1$ such that
$$
\max_{0\leq t\leq \tau}\abs{g_1(t)}=c_1 | e_m^Te^{-\tau H_m}e_1 |\
\text{and}\ \max_{0\leq t\leq \tau}\abs{g_2(t)}=c_2 | e_m^T\phi_1(-\tau H_m)e_1 |.
$$
From the above, (\ref{result1}) and (\ref{result2}), we have
\begin{equation*}
    \norm{E_m(e^z,\tau)}\leq c_1\tau
    \frac{e^{\tau\mu_2}-1}{\tau\mu_2}\xi_1\
    \text{and}\ \norm{E_m(e^z,\tau)}\leq c_2\tau (1+\gamma_2)\xi_2,
\end{equation*}
where $\mu_2=\mu_2[-A]$ and $\gamma_2$ is defined as (\ref{gamma2}).
Provided that $c_1$ or $c_2$ is not large, $\xi_1$ or $\xi_2$ is expected to
estimate the true error norm reliably. However, we should be aware that $\xi_1$
or $\xi_2$ may underestimate the true error norms considerably in the
non-Hermitian case when $c_1$ or $c_2$ is large.

The following example illustrates the behavior of $\xi_1^{rel}$ and $\xi_2^{rel}$
for $A$ non-Hermitian and their effectiveness of estimating the relative
error in (\ref{relerror}).

\emph{Example 2.} This example is taken from \cite[Example 5.3]{restarted8},
and considers the initial boundary value problem
\begin{eqnarray*}
  \dot{u}-\Delta u+\delta_1u_{x_1}+\delta_2u_{x_2}=0 & & \text{on}\ (0,\,1)^3
  \times(0,\,T), \\
  u(x,t)=0 & & \text{on}\ \partial (0,\,1)^3\ \text{for all}\ t\in[0,\,T], \\
  u(x,0)=u_0(x), & & x\in (0,\,1)^3.
\end{eqnarray*}
Discretizing the Laplacian by the seven-point stencil and the first-order
derivatives by central differences on a uniform meshgrid  with meshsize
$h=1/(n+1)$ leads to an ordinary initial value problem
\begin{eqnarray*}
  \dot{u}(t) &=& -Au(t),\ t\in(0,\,T), \\
  u(0) &=& u_0.
\end{eqnarray*}
The nonsymmetric matrix $A$ of order $N=n^3$ can be represented as the Kronecker
product form
\begin{equation*}
    A=-\frac{1}{h^2}[I_n\otimes(I_n\otimes C_1)+(B\otimes I_n+I_n\otimes C_2)
    \otimes I_n].
\end{equation*}
Here $I_n$ is the identity matrix of order $n$ and
$$
B={\rm tridiag}(1,-2,1),\ C_j={\rm tridiag}
(1+\zeta_j,-2,1-\zeta_j),\ j=1,2,
$$
where $\zeta_j=\delta_jh/2$. This is a popular test problem as
the eigenvalues of $A$ are explicitly known. Furthermore,
if $\mid\zeta_j\mid>1$ for at least one $j$, the eigenvalues
of $A$ are complex and lie in the right plane.
For the spectral properties of $A$, refer to \cite{restarted8,polykry3}.

As in \cite{restarted8}, we choose $h=1/15,\ \delta_1=96,\ \delta_2=128$,
which leads to $N=2744$ and $\zeta_1=3.2$, $\zeta_2\approx 4.27$, and approximate
$e^{-\tau A}v$ where $\tau=h^2$ and $v=[1,1,\ldots,1]^T$. We compare relative
error estimates $\xi_1^{rel}$ and $\xi_2^{rel}$ defined by (\ref{relest}) with the true
relative error defined by (\ref{relerror}). The convergence curves of these three
quantities are depicted in Figure~\ref{figure2}. Similar
to the case where $A$ is Hermitian, we observe that $\xi_1^{rel}$ and $\xi_2^{rel}$
both have excellent behavior, and they mimic the true relative error very well.
Particularly, the $\xi_2^{rel}$ are almost identical to
the true relative errors, and are more accurate than the $\xi_1^{rel}$ by
about one order.

\begin{figure}[!htb]
\centering
\includegraphics[width=7cm,height=7cm]{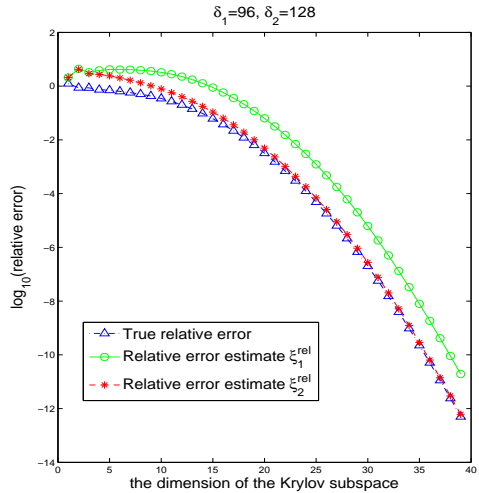}
\caption{Example 2: The relative error estimates
and the true relative error for $e^{-\tau A}v$ for $A$ nonsymmetric
with order $N=2744$.}\label{figure2}
\end{figure}

\subsection{Applications to other matrix functions}\label{subsec:5.3}

Theorem \ref{error expansion1} has indicated that the error expansion works
for all sufficiently smooth functions, so it applies to $\sin(z)$ and $\cos(z)$.
Furthermore, as analyzed and elaborated in Remark~\ref{rem3.2},
just as for $e^{-\tau A}v$, the first term of
(\ref{expansion1}) is generally a good error estimate of the Arnoldi approximation
to these matrix functions acting on a vector. We now confirm the effectiveness
of $\xi_2^{rel}$ for $\sin(A)v$ and $\cos(A)v$.
We also test the behavior of $\xi_1^{rel}$ for these two functions.

\emph{Example 3.} We consider the Arnoldi approximation to $\cos(-\tau A)v$.
Here we choose the matrix $A$ and the vector $v$ as in Example 1 for the
symmetric case and as in Example 2 for the nonsymmetric case, respectively.

\begin{figure}[!htb]
\centering
\begin{minipage}[t]{0.45\linewidth}
    \centering
    \includegraphics[width=6.5cm,height=6.5cm]{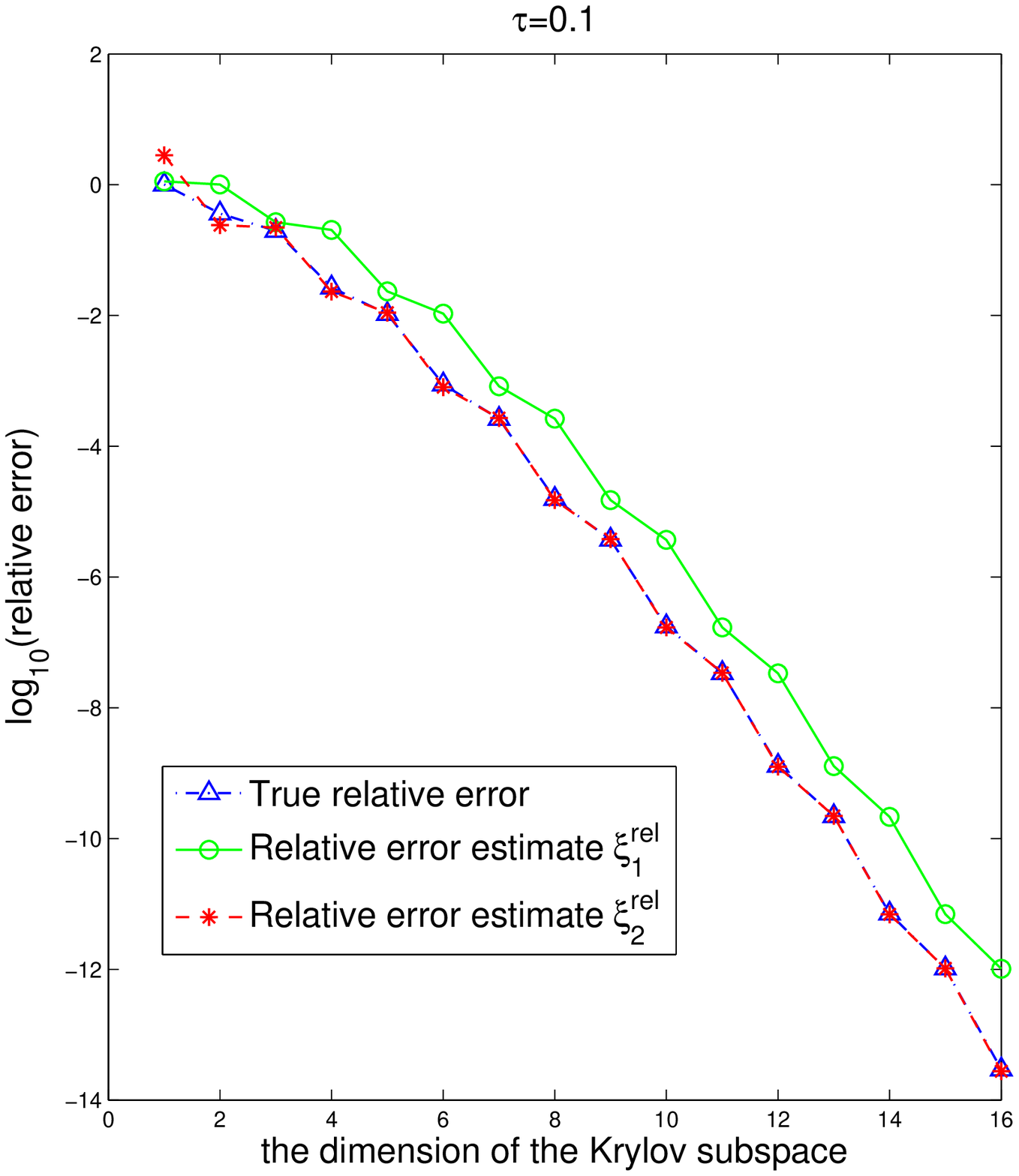}
\end{minipage}
\hspace{2ex}
\begin{minipage}[t]{0.45\linewidth}
    \centering
    \includegraphics[width=6.5cm,height=6.5cm]{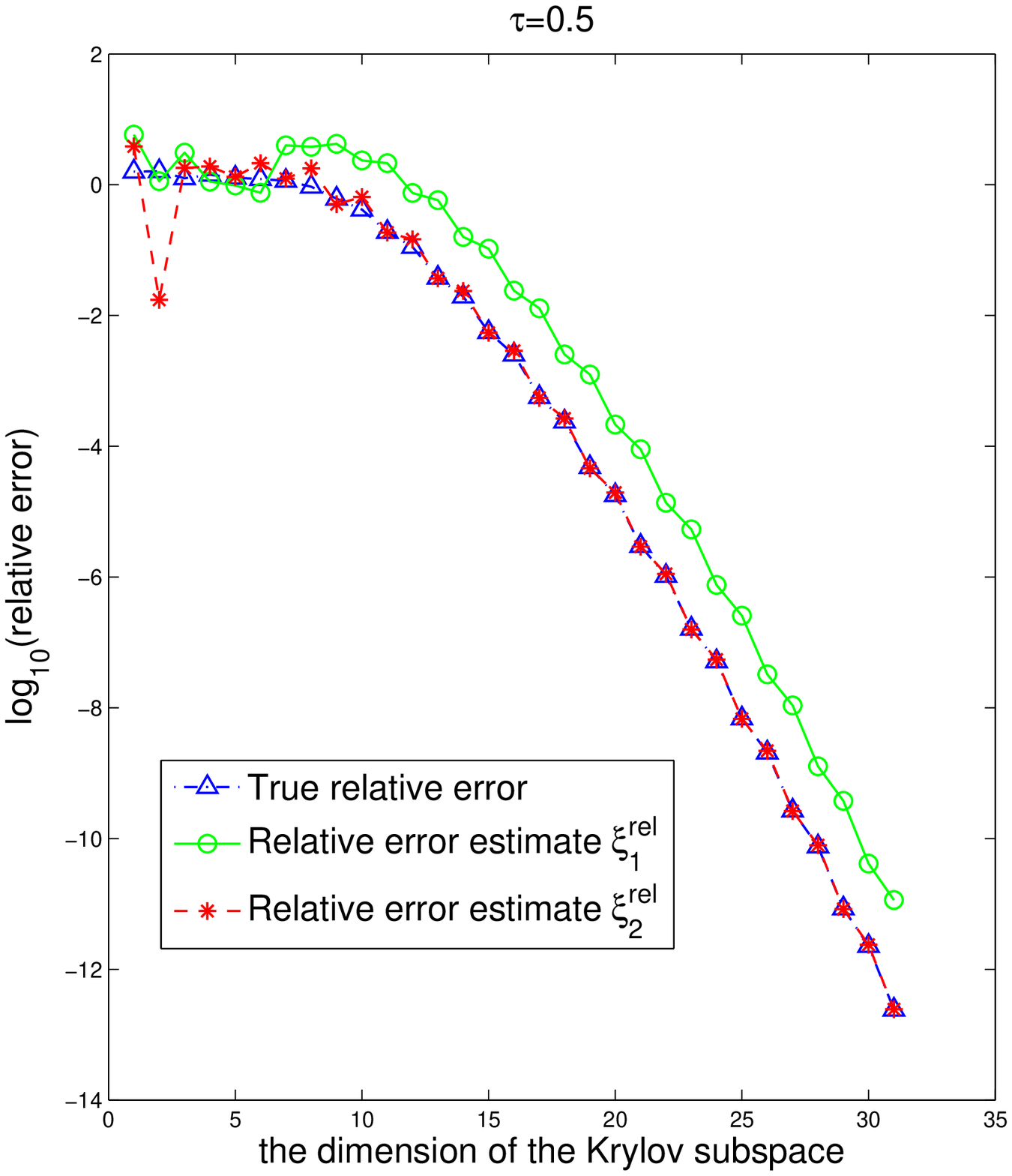}
\end{minipage}
\centering
\includegraphics[width=7cm,height=7cm]{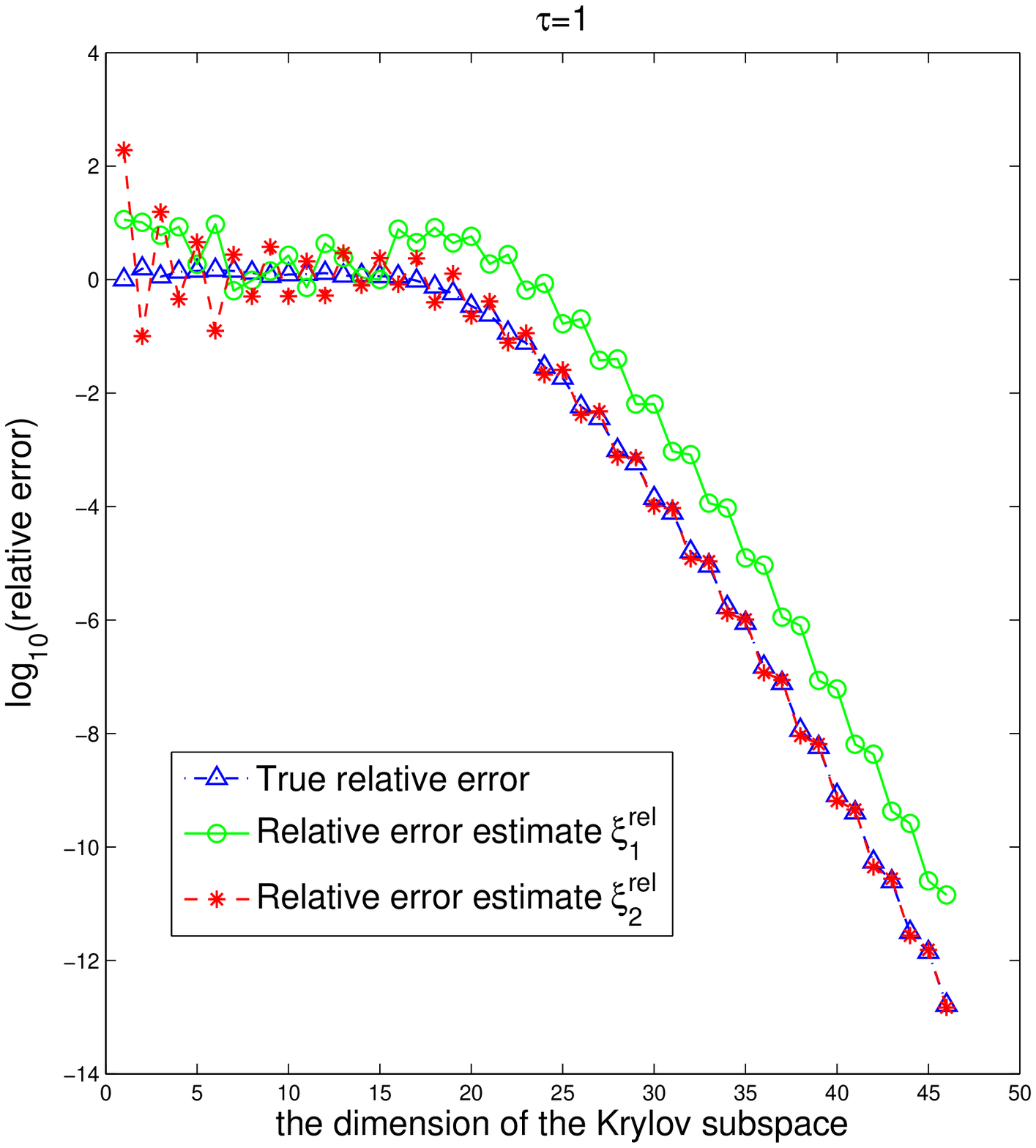}
\caption{Example 3: The relative error estimates and the true relative
error for $\cos(-\tau A)v$ for $A$ symmetric with $N=1001$}
\label{figure3}
\end{figure}

\begin{figure}[!htb]
\centering
\includegraphics[width=6.5cm,height=6.5cm]{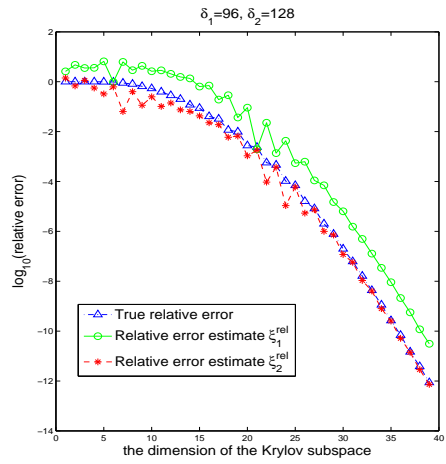}
\caption{Example 3: The relative error estimates and the true relative
error for $\cos(-\tau A)v$ for $A$ is nonsymmetric
with $N=2744$}
\label{figure4}
\end{figure}

Figures~\ref{figure3}--\ref{figure4} illustrate the behavior of
$\xi_1^{rel}$ and $\xi_2^{rel}$ when computing $\cos(-\tau A)v$
for $A$ symmetric and nonsymmetric, respectively. For $A$ symmetric,
the $\xi_2^{rel}$ are not accurate and underestimate or overestimate
the true relative errors in the first few steps for
$\tau=0.5, 1$. However, as commented in Section~\ref{subsec:5.1},
this does not cause any problem because, during this stage,
the true relative errors also stay around one and the Lanczos
approximations have no accuracy.
After this stage, $\xi_1^{rel}$ and especially $\xi_2^{rel}$ soon become smooth
and provide accurate estimates for the true relative errors as $m$ increases.
In particular, the $\xi_2^{rel}$ have little difference with the true relative
errors for the given three $\tau$. For the error estimates of the Lanczos
approximations to $\sin(-\tau A)v$, we have
very similar findings, so we do not report
the results on it.

For $\cos(-\tau A)v$ with nonsymmetric, the $\xi_1^{rel}$ and
especially $\xi_2^{rel}$ also exhibit excellent behavior and are
quite accurate to estimate the true relative error for each $m$ when the latter
starts becoming small. Moreover,
the $\xi_2^{rel}$ are more accurate than $\xi_1^{rel}$
and mimic the true relative errors very well when the
Arnoldi approximation starts converging. So we conclude that $\xi_2^{rel}$ is
very reliable to measure the true relative error of the Arnoldi
approximations to other analytic functions. For $\sin(-\tau A)v$,
we have observed very similar phenomena.

\subsection{Applications to other Krylov-like decompositions}

The error expansion in Theorem \ref{error expansion1} holds for
the Krylov-like decomposition, which includes the restarted Krylov subspace
method for approximating $f(A)v$ proposed in \cite{restarted8} as a special
case. We now confirm the effectiveness of $\xi_1^{rel}$ and $\xi_2^{rel}$ for
the restarted Krylov subspace method for approximating $e^{-\tau A}v$
and $\cos(-\tau A)v$.

\emph{Example 4.} Consider the restarted Krylov
algorithm \cite{restarted8} for approximating $e^{-\tau A}v$ and $\cos(-\tau A)v$
by choosing the matrix $A$ and
the vector $v$ as in Example 2. The method is restarted after each $m$ steps
until the true relative error in (\ref{relerror}) drops
below $\epsilon=10^{-12}$. We test the algorithm with $m=5,\,10$, respectively.
\begin{figure}[!htb]
\centering
\begin{minipage}[t]{0.45\linewidth}
    \centering
    \includegraphics[width=6.5cm,height=6.5cm]{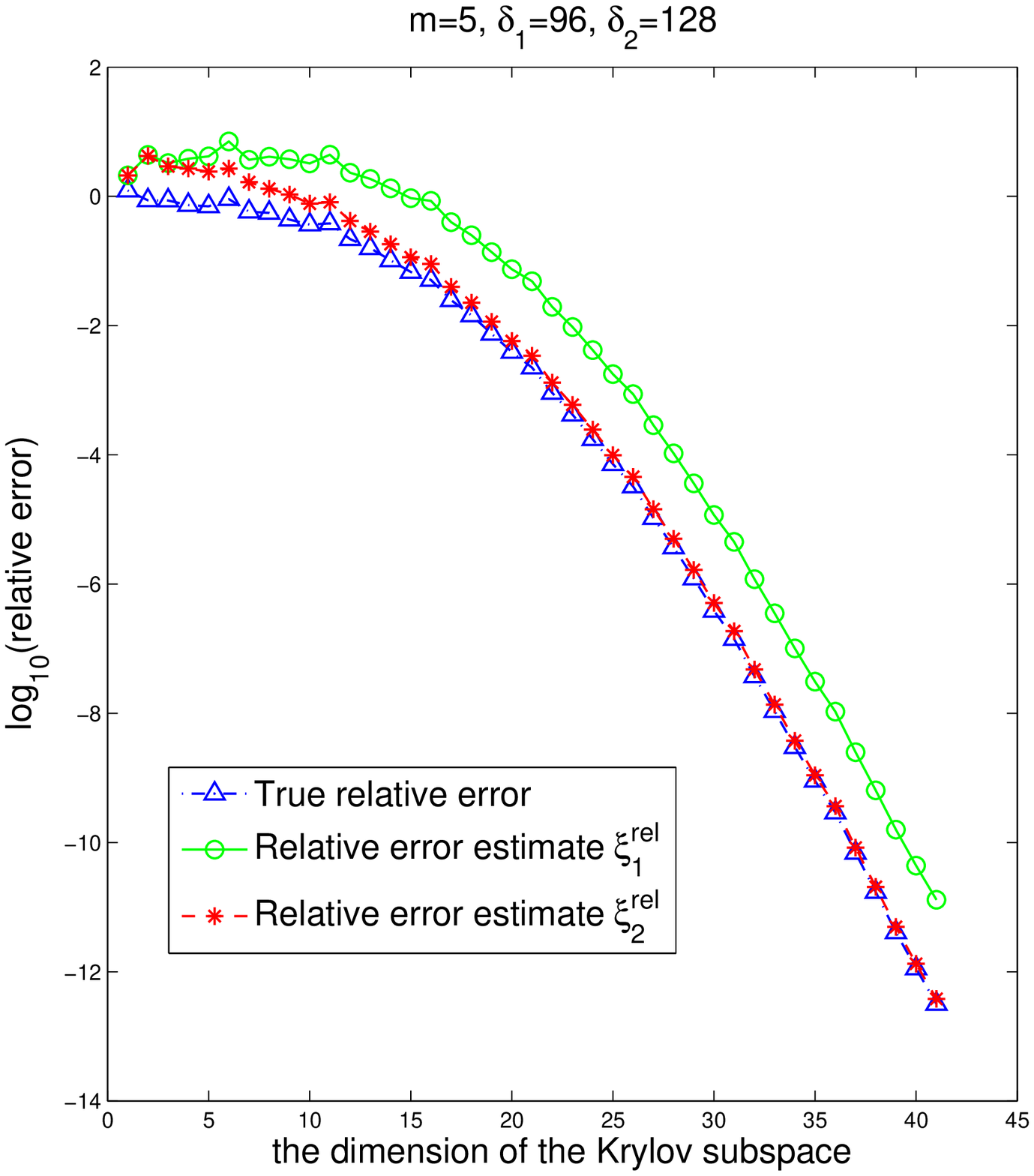}
\end{minipage}
\hspace{2ex}
\begin{minipage}[t]{0.45\linewidth}
    \centering
    \includegraphics[width=6.5cm,height=6.5cm]{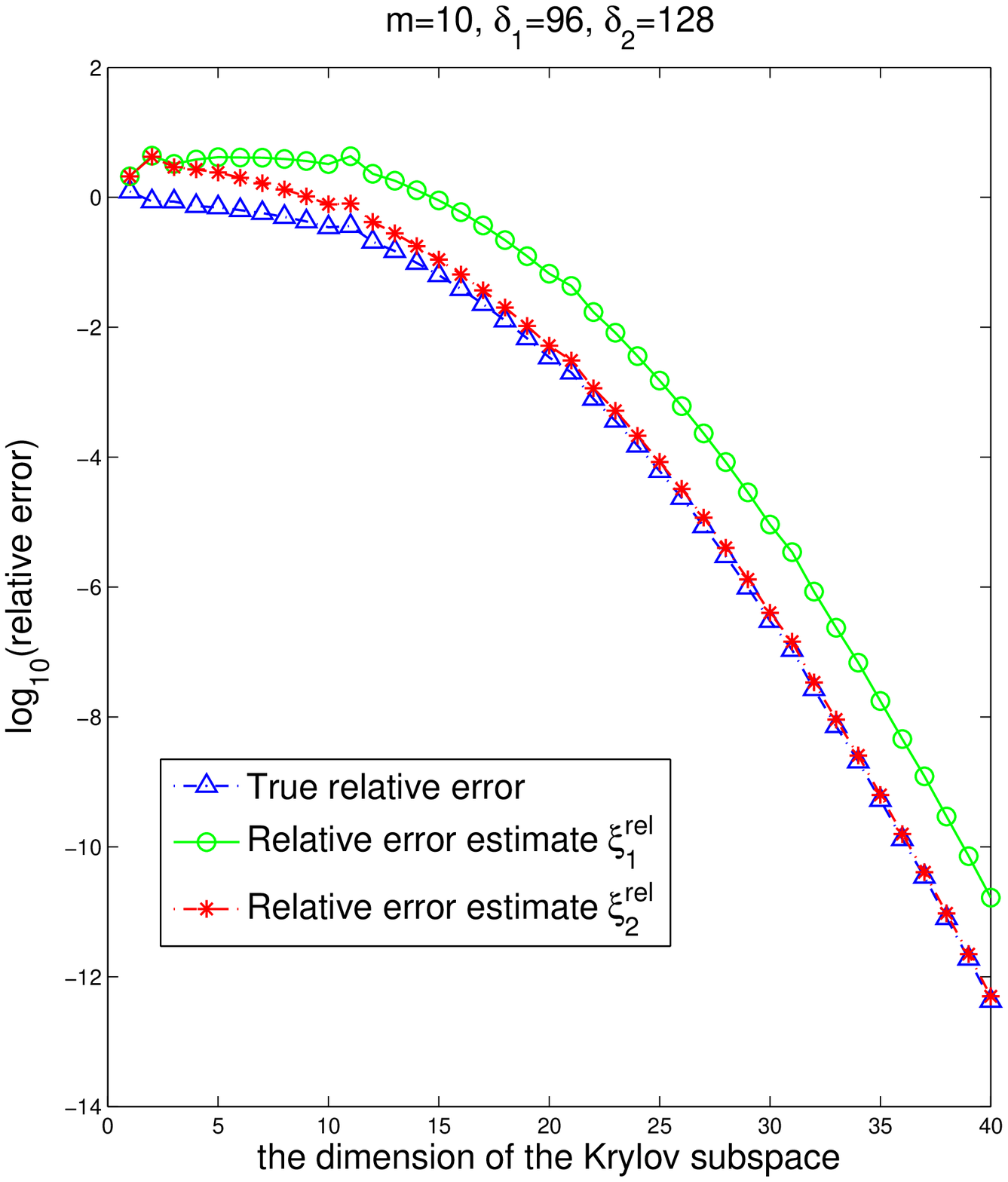}
\end{minipage}
\caption{Example 4: The relative error estimates and the true relative
error for $e^{-\tau A}v$}
\label{figure5}
\end{figure}

\begin{figure}[!htb]
\centering
\begin{minipage}[t]{0.45\linewidth}
    \centering
    \includegraphics[width=6.5cm,height=6.5cm]{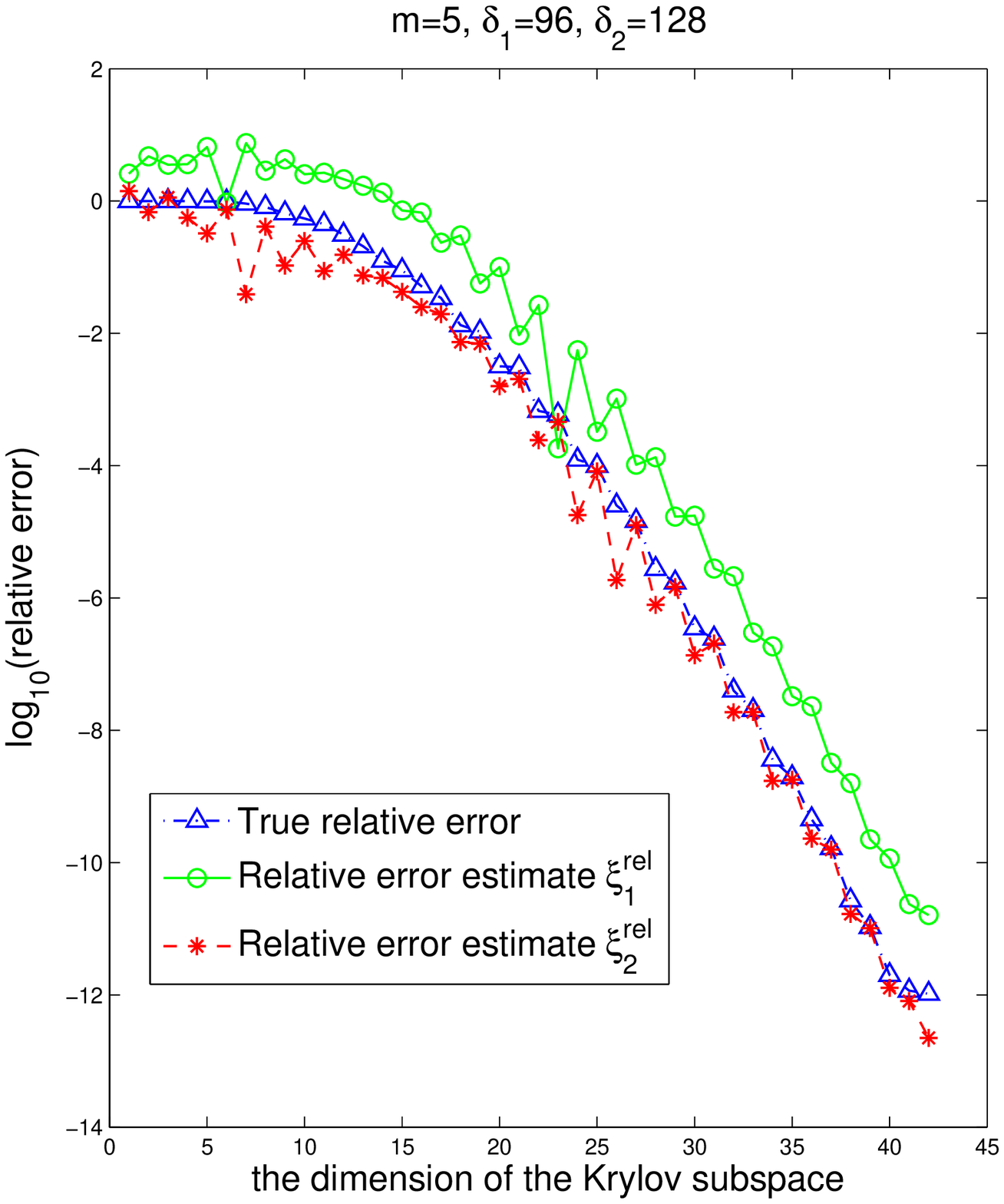}
\end{minipage}
\hspace{2ex}
\begin{minipage}[t]{0.45\linewidth}
    \centering
    \includegraphics[width=6.5cm,height=6.5cm]{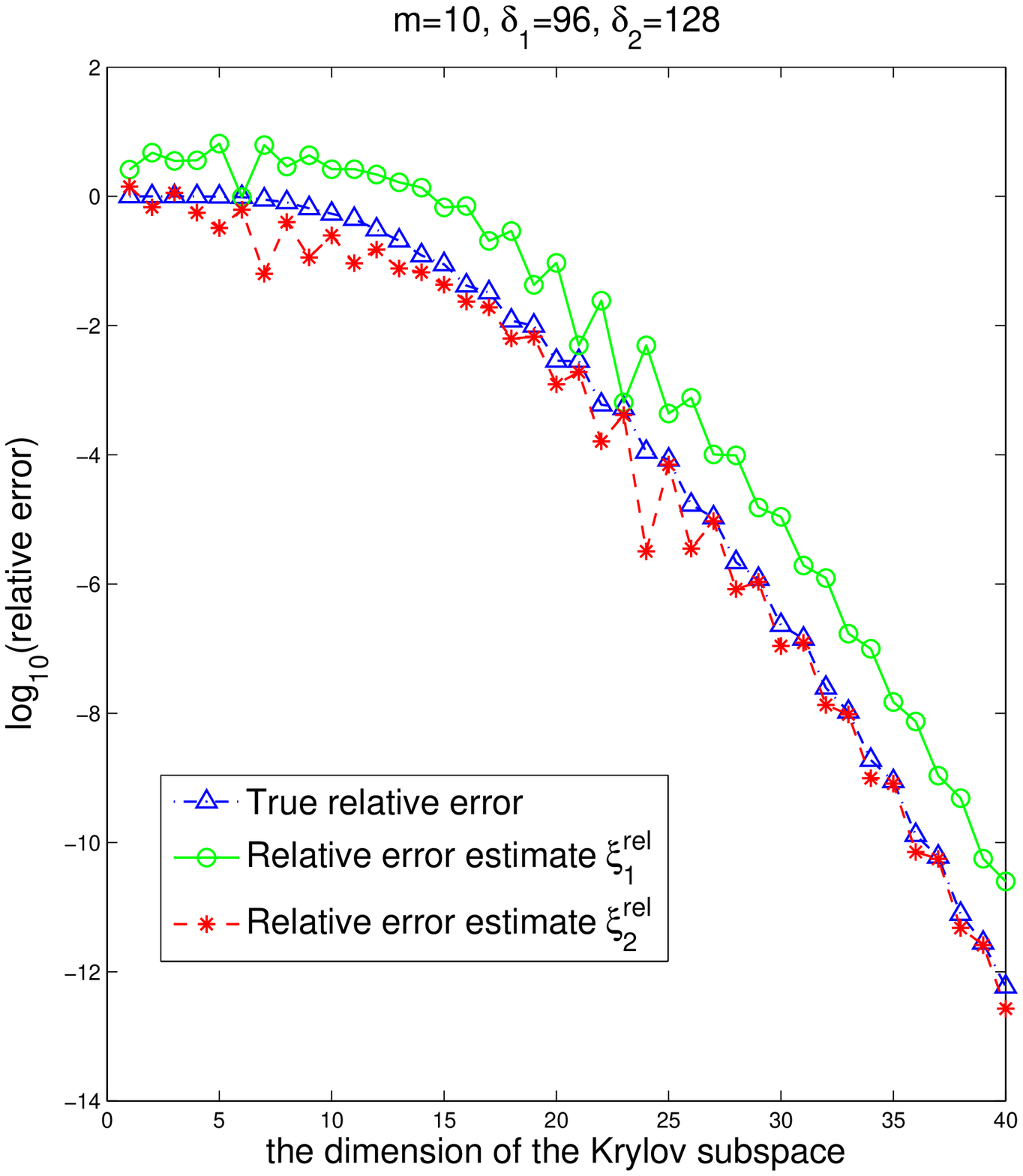}
\end{minipage}
\caption{Example 4: The relative error estimates and the true relative
error for $\cos(-\tau A)v$}
\label{figure6}
\end{figure}

Figures~\ref{figure5}--\ref{figure6} illustrate the behavior of
$\xi_1^{rel}$ and $\xi_2^{rel}$ when computing $e^{-\tau A}v$ and $\cos(-\tau A)v$
by the restarted algorithm,
where the $x$-axis denotes the sum of dimensions of Krylov subspaces.
As in the case that the (non-restarted)
Arnoldi approximations in Example 2, we observe that for approximating
$e^{-\tau A}v$ by the restarted algorithm,
both $\xi_1^{rel}$ and $\xi_2^{rel}$ exhibit excellent behavior,
and the $\xi_2^{rel}$ are almost identical to
the true relative errors, and are more accurate than the $\xi_1^{rel}$ by
about one order. For approximating $\cos(-\tau A)v$, similar to Example 3,
the $\xi_1^{rel}$ and $\xi_2^{rel}$ are still quite accurate to estimate
the true relative errors when the latter starts becoming small. We have
observed the same behavior for approximating $\sin(-\tau A)v$. Particularly,
the $\xi_2^{rel}$ are considerably better than the $\xi_1^{rel}$ and mimic
the true relative errors
very well when the restarted Arnoldi approximation starts converging.
Therefore, we may well claim that, for both non-restarted
Arnoldi approximations and the restarted Krylov-like approximations,
$\xi_2^{rel}$ is very reliable to measure the true relative error
of the Krylov-like approximations for $f(z)=e^z$ and $\cos(z),\ \sin(z)$.

\section{Conclusion} \label{sec:6}

We have generalized the error expansion of the Arnoldi approximation to
$e^Av$ to the case of Krylov-like approximations for sufficiently smooth
functions $f(z)$.
We have derived two new a priori upper bounds for the Arnoldi approximation
to $e^{-\tau A}v$ and established more compact results for $A$ Hermitian. From
them, we have proposed two practical a posteriori error estimates
for the Arnoldi and Krylov-like approximations.
For the matrix exponential, based on the new error expansion, we have
quantitatively proved that the first term of the expansion is a
reliable estimate for the whole error, which has been numerically confirmed to
be very accurate to estimate the true error. For sufficiently smooth functions
$f(z)$, we have shown why the first term of the error expansion can also
be a reliable error estimate for the whole error. We have numerically confirmed
the effectiveness of them for the cosine and sine matrix functions
frequently occurring in applications. It is worthwhile to point out
that $\xi_2$ is experimentally more accurate than $\xi_1$ for the exponential,
cosine and sine functions.

We have experimentally found that, for $A$ Hermitian, the reliability of
$\xi_1^{rel}$ and $\xi_2^{rel}$ is not affected by
$\gamma_1$ and $\gamma_3$, respectively, and they
are equally accurate to mimic the true relative error for
greatly varying $\gamma_1$ and $\gamma_3$. Therefore, we conjecture
that $\gamma_1$ and $\gamma_3$ can be replaced by some other better
forms, at least for $A$ real symmetric semipositive or negative,
whose sizes, unlike $e^{\tau (b-a)}$,
vary slowly with $\tau$ increasing and the spectrum of
$A$ spreading.

\bigskip

{\bf Acknowledgements}. We thank two referees very much for their
very careful reading, valuable comments and constructive suggestions,
which enable us to improve the presentation considerably.

\small
\bibliographystyle{siam}

\end{document}